\documentclass{amsart}
 \usepackage[foot]{amsaddr}

\title{Compositions of Bely\u{i} Maps and their Extended Monodromy Groups}
\author{Jacob Bond}
\address{General Motors Company}
\email{jacob.bond@gm.com}

\usepackage{amssymb}
\usepackage{textcomp}
\usepackage{hyperref}
\usepackage[mathscr]{eucal}
\usepackage{graphicx}
\usepackage{subcaption}
\usepackage{tikz-cd}
\usepackage{tkz-graph}
\usepackage{centernot}

\usepackage[style=numeric]{biblatex}
\usepackage[noend]{algpseudocode}
\usepackage{algorithm}

\algrenewcommand{\Function}[2]{\State{{\bf function} \textproc{#1}(#2)}}
\algrenewcommand{\algorithmiccomment}[1]{\hspace{\algorithmicindent}\(\triangleright\) #1}

\usepackage{pgfplots}
\pgfplotsset{compat=1.14}
\usetikzlibrary{arrows.meta}

\usepackage{color}
\definecolor{slayter}{RGB}{110, 153, 180}

\newtheorem*{defn}{Definition}
\newtheorem{thm}{Theorem}
\numberwithin{thm}{section}
\newtheorem{prop}[thm]{Proposition}
\newtheorem{lem}[thm]{Lemma}
\newtheorem{conv}[thm]{Convention}
\newtheorem{const}[thm]{Construction}
\newtheorem{cor}[thm]{Corollary}
\newtheorem{defnnum}[thm]{Definition}
\newtheorem{desc}[thm]{Description}

\newenvironment{example}[1][]{\refstepcounter{thm}\par
   \noindent\textbf{Example~\thethm. #1} \rmfamily}{}

\DeclareMathOperator{\Mon}{Mon}
\DeclareMathOperator{\EMon}{EMon}
\DeclareMathOperator{\Fun}{Fun}
\DeclareMathOperator{\id}{id}
\DeclareMathOperator{\sgn}{sgn}
\DeclareMathOperator{\re}{Re}
\DeclareMathOperator{\proj}{proj}
\DeclareMathOperator{\MonExt}{MonExt}

\newcommand{\coloneqq}{\mathrel{\rlap{%
                     \raisebox{0.5ex}{$\cdot$}}%
                     \raisebox{-0.5ex}{$\cdot$}}%
                     =}

\begin{document}

\begin{abstract}
Given a composition of Bely\u{\i} maps \(\beta \circ \gamma: X \rightarrow Z\), paths between edges of \(\beta\) are extended to form loops, then lifted by \(\gamma\).  These liftings are then studied to understand how loops in \(Z\) act on edges of \(\beta \circ \gamma\), demonstrating the group operation in \(\Mon \beta \circ \gamma \unlhd \Mon \gamma \wr \Mon \beta\).  Abstracting away the specific Bely\u{\i} map \(\gamma\) and finding the image of \(\pi_{1}(Z)\) in \(\pi_{1}(Y) \wr \Mon \beta\) instead allows subsequently determining \(\Mon \beta \circ \gamma\), for any \(\gamma\), using only the monodromy representation of \(\gamma\).
\end{abstract}

\maketitle

\section{Introduction}

A Bely\u{\i} map is a meromorphic function \(\beta\) from a Riemann surface \(X\) to \(\mathbb{P}^{1}(\mathbb{C})\) which is unbranched outside of \(\{0, 1, \infty\}\).  The pair \((X, \beta)\) is called a Bely\u{\i} pair.  Let
\[\mathbb{P}^{1}(\mathbb{C})_{*} \coloneqq \mathbb{P}^{1}(\mathbb{C}) \backslash \{0, 1, \infty\},\qquad X_{*} \coloneqq X \backslash \beta^{-1}(\{0, 1, \infty\}).\]
Associated to \((X, \beta)\) is its monodromy representation \(\rho: \pi_{1}(\mathbb{P}^{1}(\mathbb{C})_{*}, z) \rightarrow S_{d}\), where \(z \in \mathbb{P}^{1}(\mathbb{C})\) is a fixed base point and \(d\) is the degree of \(\beta\).  Let \(\Mon \beta\) denote the monodromy group of \(\beta\), which is defined as the image of \(\rho\).

Historically, the monodromy group was considered to carry the important group theoretic information about the ramification of the Bely\u{\i} map.  However, the monodromy group alone is insufficient for determining the monodromy of a composition of Bely\u{\i} maps.  For example, let
\[\beta_{1}(x) = x^{3},\qquad
\beta_{2}(x) = (1-x)^{3},\qquad
\gamma(x) = x^{2}.
\]
Then \(\Mon \beta_{1} \approx \Mon \beta_{2} \approx C_{3}\), but
\[\Mon (\beta_{1} \circ \gamma) \approx C_{6}\ (\mathrm{6T1}) \not\approx C_{2} \times A_{4}\ (\mathrm{6T6}) \approx \Mon (\beta_{2} \circ \gamma).\]
This shows that although the monodromy groups of \(\beta_{1}\) and \(\beta_{2}\) are isomorphic, composing each \(\beta_{i}\) with \(\gamma\) results in distinct monodromy groups.

To this end, the extended monodromy group will be introduced for Bely\u{\i} maps satisfying \(\beta(\{0, 1, \infty\}) \subseteq \{0, 1, \infty\}\).
An object called the extending pattern of \(\beta\) will be defined which will provide a map, the extended monodromy representation,
\[\pi_{1}(\mathbb{P}^{1}(\mathbb{C})_{*}, z) \rightarrow \pi_{1}(X_{*}, x) \wr_{E_{\beta}} \Mon \beta,\]
where \(E_{\beta}\) is the set of edges of \(\beta\) and \(x \in X\) is a fixed base point.  The extended monodromy group, denoted \(\EMon \beta\), will then be defined as the image of the extended monodromy representation.  Note that \(\Mon \beta\) can be recovered from \(\EMon \beta\) through projection onto the second component of \(\EMon \beta\).

Finally, for any Bely\u{\i} map \(\gamma\), \(\Mon \beta \circ \gamma\) is easily recovered from \(\EMon \beta\) through postcomposition of the first component of \(\EMon \beta\) by the monodromy representation of \(\gamma\).  In particular, it will be shown that given the extending pattern of \(\beta\), as well as the monodromy representations of \(\beta\) and \(\gamma\), the monodromy representation of \(\beta \circ \gamma\) can be efficiently determined.  Although not effectively computable for larger examples, a group theoretic description of the extended monodromy group is also given.

\section{Background}

Consideration of compositions of Bely\u{\i} maps has many motivations in light of their correspondence with dessins d'enfants, called dessins for short.  Shabat \& Zvonkin \cite[Example 6.1]{zvonkinshabat} refer to composition as ``a manifestation of the {\it hidden symmetries}'' of a dessin.  When a dessin decomposes as a composition of dessins, the Bely\u{\i} map of the composition can be determined by computing the Bely\u{\i} map of each piece of the composition.  In the opposite direction, increasingly complex pairs of Bely\u{\i} maps and dessins can be established through composition of simpler pairs.  Bely\u{\i} maps have even been considered for use in cryptography by way of composition.\ \cite[Chapter 5]{bond}

\subsection{Setup}\hspace*{\fill}

The rigid nature of Bely\u{\i} maps means that compositions of Bely\u{\i} maps do not always result in another Bely\u{\i} map.  However, in the case that \(\beta(\{0, 1, \infty\}) \subseteq \{0, 1, \infty\}\) for a Bely\u{\i} map \(\beta\), then for any Bely\u{\i} map \(\gamma\), \(\beta \circ \gamma\) is again a Bely\u{\i} map.

\begin{defn}
    \cite[Section 2.5.5]{landozvonkin}
    A Bely\u{\i} map \(\beta: \mathbb{P}^{1}(\mathbb{C}) \rightarrow \mathbb{P}^{1}(\mathbb{C})\) is a dynamical Bely\u{\i} map if \(\beta(\{0, 1, \infty\}) \subseteq \{0, 1, \infty\}\).
\end{defn}

\begin{lem}
    \cite[cf. Prop. 2.5.17]{landozvonkin}
    If \(\gamma: X \rightarrow \mathbb{P}^{1}(\mathbb{C})\) is a Bely\u{\i} map and \(\beta\) is a dynamical Bely\u{\i} map, then \(\beta \circ \gamma\) is a Bely\u{\i} map.
\end{lem}

In the case of a dynamical Bely\u{\i} map \(\beta: \mathbb{P}^{1}(\mathbb{C}) \rightarrow \mathbb{P}^{1}(\mathbb{C})\), let the domain of \(\beta\) be denoted by \(Y\) and the codomain by \(Z\).  Further, fix basepoints \(y \in Y\) and \(z \in Z\) and let \(\pi_{1}^{Y} \coloneqq \pi_{1}(Y, y)\) and \(\pi_{1}^{Z} \coloneqq \pi_{1}(Z, z)\).

\stepcounter{thm}
\begin{figure}[ht]
\centering
    \begin{tabular}{ccccc}
        \raisebox{-0.5\height}{\includegraphics[scale=0.8]{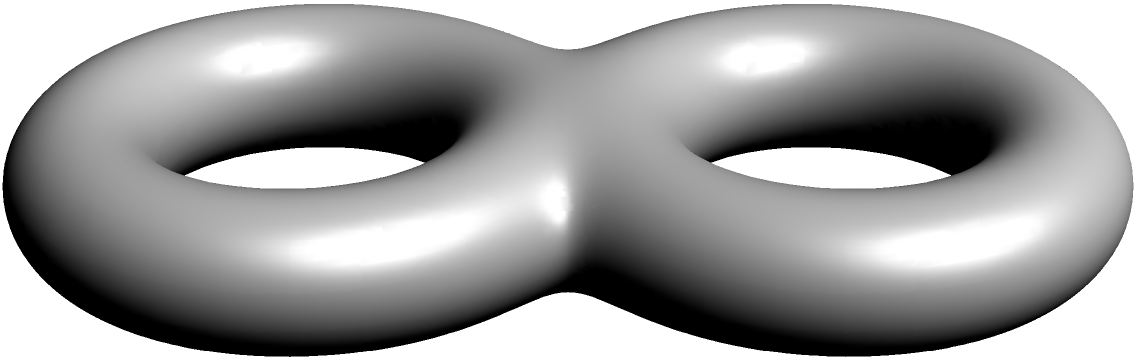}} &
        {
            \begin{tabular}{c}
                {\LARGE \(\gamma\)}\\
                \scalebox{3.5}[1.5]{\(\rightarrow\)}
            \end{tabular}
        } &
        \raisebox{-0.5\height}{\includegraphics[scale=0.45]{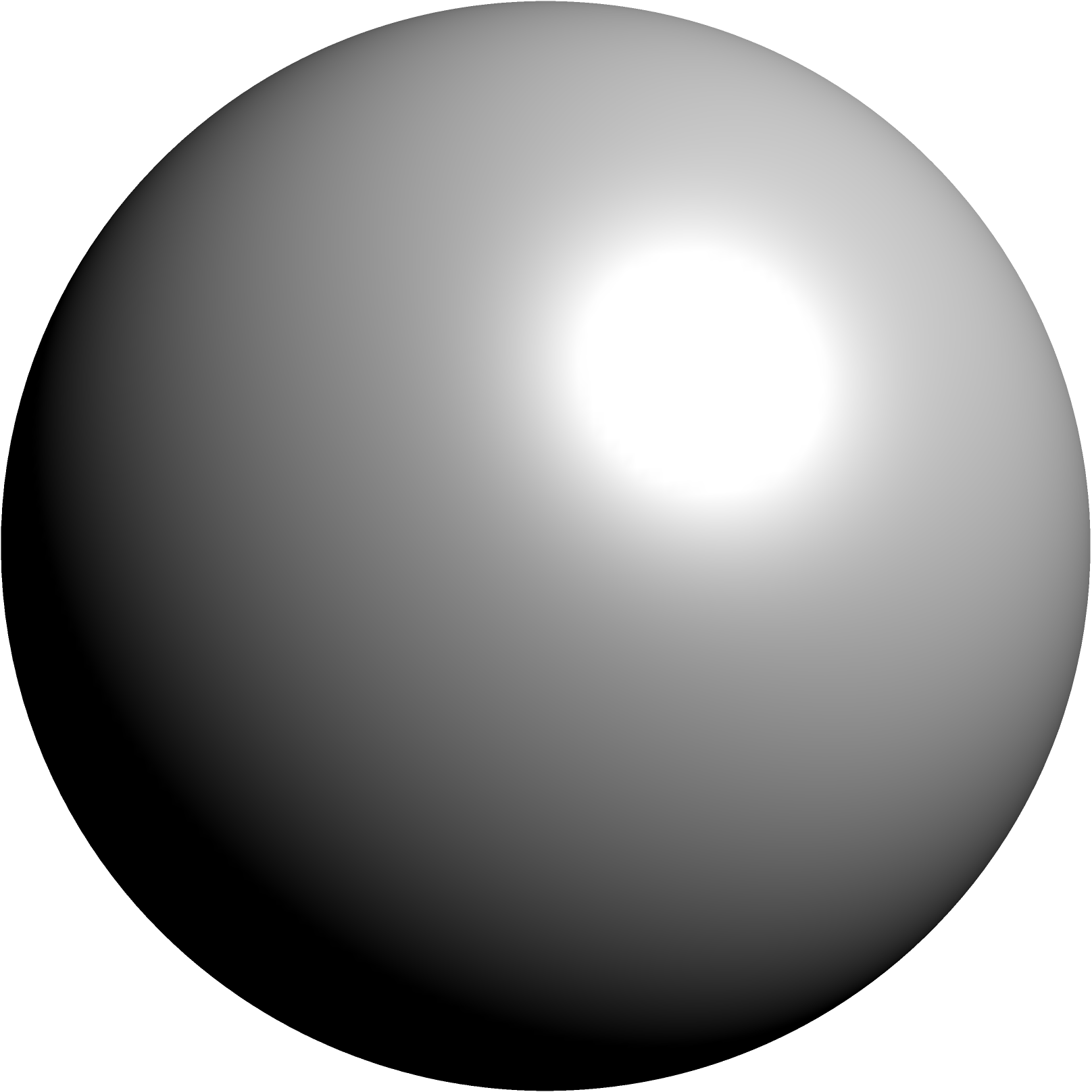}} &
        {
            \begin{tabular}{c}
                {\LARGE \(\beta\)}\\
                \scalebox{3.5}[1.5]{\(\rightarrow\)}
            \end{tabular}
        } &
        \raisebox{-0.5\height}{\includegraphics[scale=0.45]{images/sphere.png}} \\
        \raisebox{-0.75\height}{\Large\(X\)} && \raisebox{-0.75\height}{\Large\(Y\)} && \raisebox{-0.75\height}{\Large\(Z\)}
    \end{tabular}
\caption{A composition of covering maps}
\end{figure}

Associated to a Bely\u{\i} map \(\gamma\) is a dessin d'enfant, \(\Delta_{\gamma}\), obtained from the preimage of the interval \([0, 1]\) by neglecting the complex structure of the preimage.

\begin{defn}
A dessin d'enfant is a graph with a fixed bipartite structure and a labelling of the edges which specifies a cyclic ordering of the edges around each vertex.
\end{defn}

\noindent The correspondence of a preimage \(\gamma^{-1}([0, 1])\) with a dessin is made by associating \(\gamma^{-1}\big(\,(0, 1)\,\big)\) with the edges of the dessin and the points \(\gamma^{-1}(\{0, 1\})\) with the bipartition of the vertices according to whether a point lies over \(0\) or \(1\).  The cyclic ordering of the edges around the vertices in each subset of the bipartition yields a pair of permutations \(\sigma_{0}, \sigma_{1}\), arising from the vertices lying over \(0\) and \(1\), respectively.  Setting \(\sigma_{\infty} \coloneqq (\sigma_{1}\sigma_{0})^{-1}\) produces a 3-constellation \(\{\sigma_{0}, \sigma_{1}, \sigma_{\infty}\}\):  a triple that acts transitively on \(\{1, \ldots, n\}\) and whose product is the identity.\ \cite[Section 1.1]{landozvonkin}

\bigskip\noindent
{\it Notation.} Group actions will be applied on the right as \(x^{g}\).  This agrees with the standard notation for path composition where \(p_{1} \ast p_{2}\) traverses \(p_{1}\) followed by \(p_{2}\).  Additionally, computer algebra systems, such as SageMath \cite{sage} and GAP \cite{gap}, multiply permutations from left to right.

\subsection{Related Work}\hspace*{\fill}

Shabat \& Zvonkin \cite[Section 6]{zvonkinshabat} observe that the operation of composition of two plane trees \(\Delta_{\beta}, \Delta_{\gamma}\) is tantamount to substituting \(\Delta_{\beta}\) for every edge of \(\Delta_{\gamma}\).  Adrianov \& Zvonkin \cite[Theorem 3.3]{zvonkinadrianov} refine this notion, obtaining a description of the permutations \(\sigma_{0}, \sigma_{1}\) defining the plane tree \(\Delta_{\beta \circ \gamma}\), although they stop short of identifying the group \(\langle \sigma_{0}, \sigma_{1}\rangle\) or determining when it is a proper subgroup of \(\Mon \beta \wr \Mon \gamma\).  Lando \& Zvonkin \cite[Prop. 1.7.10]{landozvonkin} reiterate that the monodromy group ``can be represented as a subgroup of'' \(\Mon \beta \wr \Mon \gamma\).  Wood \cite[Section 3.3]{wood} obtains a similar description as \cite{zvonkinadrianov} for \(\sigma_{0}, \sigma_{1}\) in the case of Bely\u{\i} maps defined over \(\mathbb{R}\) through a geometric approach.

\subsection{Organization}\hspace*{\fill}

In Section 3, the extending pattern of a Bely\u{\i} map will be established by tracking the way in which it lifts the canonical triangulation \cite[Section 1.5.4]{landozvonkin} and the way in which the preimages of \([0, 1]\) intersect the lifted structure.  Section 4 will use the extending pattern of a Bely\u{\i} map to construct the extended monodromy group \(\EMon \beta\).  Having established \(\EMon \beta\), Section 5 will determine generators for \(\Mon \beta \circ \gamma\), and Section 6 will describe the structure of \(\Mon \beta \circ \gamma\) as a subgroup of \(\Mon \gamma \wr_{E_{\beta}} \Mon \beta\), concluding with an example.

\section{Constructing the Extending Pattern}

Being an unramified covering map away from \(\{0, 1, \infty\}\), the preimage of a Bely\u{\i} map \(\gamma^{-1}(\mathbb{P}^{1}(\mathbb{C})_{*})\) decomposes into disjoint sheets \(\{\mathcal{S}_{i}\}_{i = 1}^{n}\), where \(\deg \gamma = n\).  The approach taken in determining the monodromy of a composition uses sheets of covering maps to track traversals around the points \(0\), \(1\), and \(\infty\).  For this reason, it is important to have a precise specification of the sheets lying over \(\mathbb{P}^{1}(\mathbb{C})\), which also enables a correspondence between the edges \(E_{\beta \circ \gamma}\) of a composition and the Cartesian product of edges of each function \(E_{\beta} \times E_{\gamma}\).\ \cite[Prop. 3.5]{wood}

Key to the determination of the monodromy group of a composition will be consideration of the effect traversing paths between edges of \(\beta\) has on the sheets of \(\gamma\) when the paths are lifted by \(\gamma\).  To evaluate this effect, an object referred to as the extending pattern, which is constructed by examining the paths between edges of \(\beta\), will be utilized.

\subsection{Covering Maps, Sheets, \& Edges}\hspace*{\fill}

For notational convenience, let \([-\infty, 0]\) denote the nonpositive real axis together with \(\infty\) as a subset of \(\mathbb{P}^{1}(\mathbb{C})\).  Further, let \(\mathbb{H}\) denote the upper half-plane of \(\mathbb{C}\) embedded into \(\mathbb{P}^{1}(\mathbb{C})\) and \(\overline{\mathbb{H}}\) denote the closure of \(\mathbb{H}\) in \(\mathbb{P}^{1}(\mathbb{C})\).

Let \(\{B_{i}\}_{i = 1}^{n}\) be the path components of
\[\gamma^{-1}\Big(\mathbb{P}^{1}(\mathbb{C}) \backslash ([-\infty, 0] \cup [1, \infty]) \Big).\]
Let \(\{\mathscr{T}_{i}\}_{i = 1}^{n}\) be the path components of \(\gamma^{-1}\big(\overline{\mathbb{H}}\big)\), ordered so that \(\mathscr{T}_{i}\) is the unique path component of \(\gamma^{-1}(\overline{\mathbb{H}})\) intersecting \(B_{i}\).  Then \(\gamma\) maps \(B_{i} \cup \mathscr{T}_{i}\) homeomorphically onto \(\mathbb{P}^{1}(\mathbb{C})\).\ \cite[Prop. 13.3]{bump}

\begin{conv}
The sheets of the Bely\u{\i} map \(\gamma\), regarded as a covering map of \(\mathbb{P}^{1}(\mathbb{C})_{*}\), will be defined to be \(\big\{(B_{i} \cup \mathscr{T}_{i}) \backslash \{0, 1, \infty\}\big\}_{i = 1}^{n}\), as presented above.
\label{sheetconv}
\end{conv}

The various sets just introduced are illustrated in Figure \ref{fundamentaldomain}, in which \begin{itemize}
    \item a set \(B_{i}\) is the interior of a pair of dark and light triangles together with the edge of the dessin between them, 
    \item a set \(\mathscr{T}_{i}\) is a dark triangle together with its boundary, and
    \item a sheet of \(\gamma\) is a set \(B_{i}\) together with the boundary of dark triangle contained in \(B_{i}\).
\end{itemize}

\stepcounter{thm}
\begin{figure}[ht]
    \centering
    \includegraphics[scale=0.55]{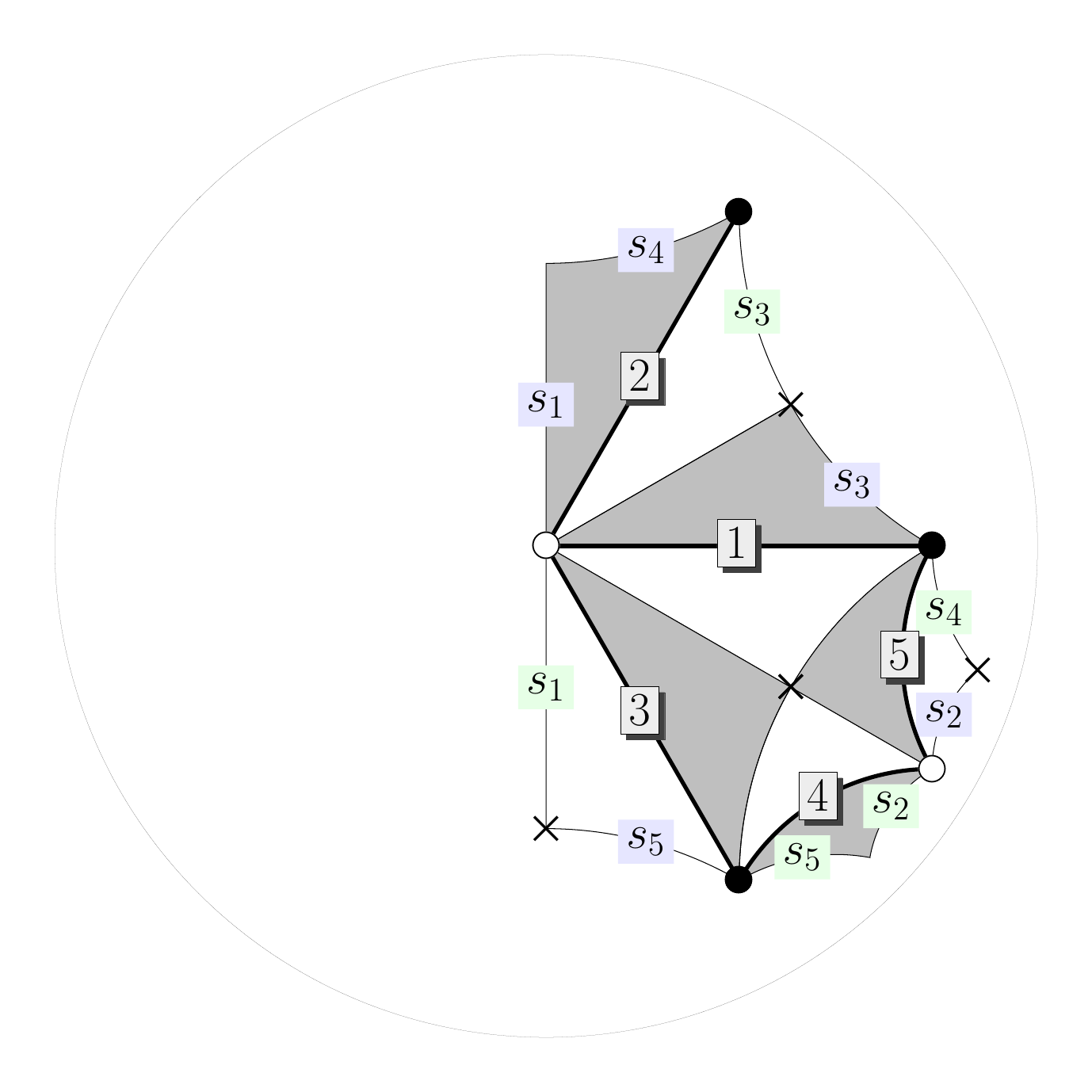}
    \caption{A fundamental domain for a dessin \(\Delta\).  Image from \cite{belyidb,voightetal}.}
    \label{fundamentaldomain}
\end{figure}

\noindent Note that this convention is related to the concept of the canonincal triangulation \cite[Section 1.5.4]{landozvonkin} and coincides with the approach used by Wood \cite{wood}.

The edges of the dessins and their Bely\u{\i} maps will play a critical role, so that it is important to have an explicit definition for the edges of a Bely\u{\i} map.  Let \(\mathbb{I} = [0, 1] \subseteq \mathbb{R}\), which, by abuse of notation, may also be considered as a subset of \(\mathbb{C}\) or \(\mathbb{P}^{1}(\mathbb{C})\).  Analogously, let \(\mathbb{I}^{\circ}\) denote the interior of \(\mathbb{I}\) in \(\mathbb{R}\), though it may be embedded in \(\mathbb{P}^{1}(\mathbb{C})\) as well.
\begin{defnnum}
    An edge of a Bely\u{\i} map \(\gamma\) is a lifting of \(\id_{\mathbb{I}}: x \mapsto x\) by \(\gamma\).  That is, an edge of \(\gamma\) is a continuous function \(e: \mathbb{I} \rightarrow \gamma^{-1}(\mathbb{I})\) satisfying \(\gamma \circ e = \id_{\mathbb{I}}\).
\end{defnnum}
\noindent An edge will often be implicitly identified with its image in \(\mathbb{P}^{1}(\mathbb{C})\).

Let \(E_{\gamma}\) denote the set of edges of \(\gamma\).  Since each sheet \(\mathcal{S}_{i}\) of \(\gamma\) maps homeomorphically onto \(\mathbb{P}^{1}(\mathbb{C})\), there is a bijection between edges \(E_{\gamma}\) and sheets \(\{\mathcal{S}_{i}\}\).

\begin{prop}
    There is a bijection \(E_{\beta \circ \gamma} \leftrightarrow E_{\gamma} \times E_{\beta}\) defined as follows:
    
    \begin{center}
    For an edge \(e \in E_{\beta \circ \gamma}\), let \(\mathcal{S}_{e}\) be the sheet of \(\gamma\) containing \(e(1/2)\) and\newline let \(e_{\gamma}\) be the unique element of \(E_{\gamma}\) with \(e_{\gamma} \in \mathcal{S}_{e}\).  Then
    \[e \longleftrightarrow \big(e_{\gamma}, \gamma(e)\big).\]
    \end{center}
    \label{edgebijection}
\end{prop}

\bigskip\noindent
{\it Note.} There is nothing special about the point \(1/2\) in Proposition \ref{edgebijection}, as any point in the interval \((0, 1)\) would suffice.  Its sole purpose is to assign a unique sheet to each edge \(e_{\beta \circ \gamma} \in E_{\beta \circ \gamma}\).

\begin{proof}
Given an edge \(e_{\beta \circ \gamma}\) of \(\beta \circ \gamma\), \(\gamma(e_{\beta \circ \gamma})\) is an edge of \(\beta\) as
\[\beta \circ \gamma(e_{\beta \circ \gamma}) = \id_{\mathbb{I}} \Longrightarrow \beta(\gamma \circ e_{\beta \circ \gamma}) = \id_{\mathbb{I}}.\]

To see that the mapping is surjective, let \((e_{\gamma}, e_{\beta}) \in E_{\gamma} \times E_{\beta}\).  Then \(\gamma^{-1}(e_{\beta})\) consists of \(\deg \gamma\) distinct preimages of \(e_{\beta}\), each constituting an edge of \(\beta \circ \gamma\).  In particular, each of the \(\deg \gamma\) points in \(A \coloneqq \gamma^{-1}\big(e_{\beta}(1/2)\big)\) is a preimage of \(1/2\) by \(\beta \circ \gamma\), as
\[\beta \circ \gamma \left( \gamma^{-1} \circ e_{\beta}\left(\frac{1}{2}\right)\right) = \frac{1}{2}.\]
Finally, each point of \(A\) must lie in a distinct sheet of \(\gamma\) because \(\gamma\) maps each sheet injectively onto \(\mathbb{P}^{1}(\mathbb{C})\), so that exactly one point \(a\) of \(A\) lies in the sheet containing \(e_{\gamma}\).  The edge \(e\) of \(E_{\beta \circ \gamma}\) containing \(a\) maps to \((e_{\gamma}, e_{\beta})\).

By the cardinality of the involved sets from degree considerations, the mapping is a bijection.
\end{proof}

The bijection of Proposition \ref{edgebijection} is illustrated in Figure \ref{bijectionfig} for
\begin{gather}
\beta(x) = \big(1 - \mu(x)\big)^{3}, \qquad \gamma(x) = 3x^{2} - 2x^{3},\\
\mu(x) = \frac{-21x}{5i\sqrt{3}(x-1) - 11x-10}.\nonumber
\label{betadef}
\end{gather}
Henceforth, \(E_{\beta \circ \gamma}\) will be identified with \(E_{\gamma} \times E_{\beta}\) by this bijection.

\stepcounter{thm}
\begin{figure}[ht]
    \centering
    \subcaptionbox{A dynamical Bely\u{\i} map \(\beta\)}{%
        \includegraphics[width=0.475\textwidth]{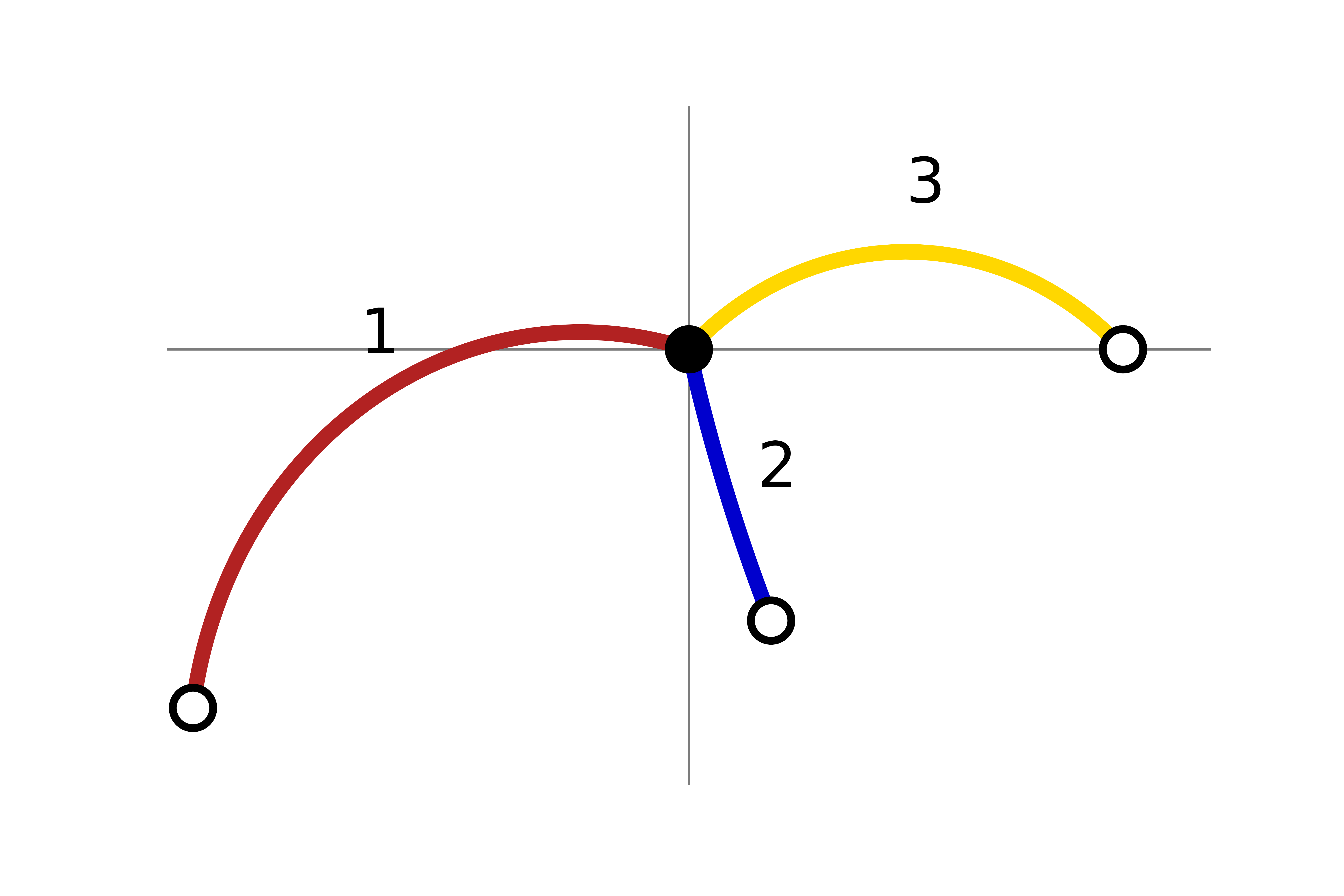}}\hspace{0.01\textwidth}%
    \subcaptionbox{A Bely\u{\i} map \(\gamma\) and its sheets}{%
        \includegraphics[width=0.475\textwidth]{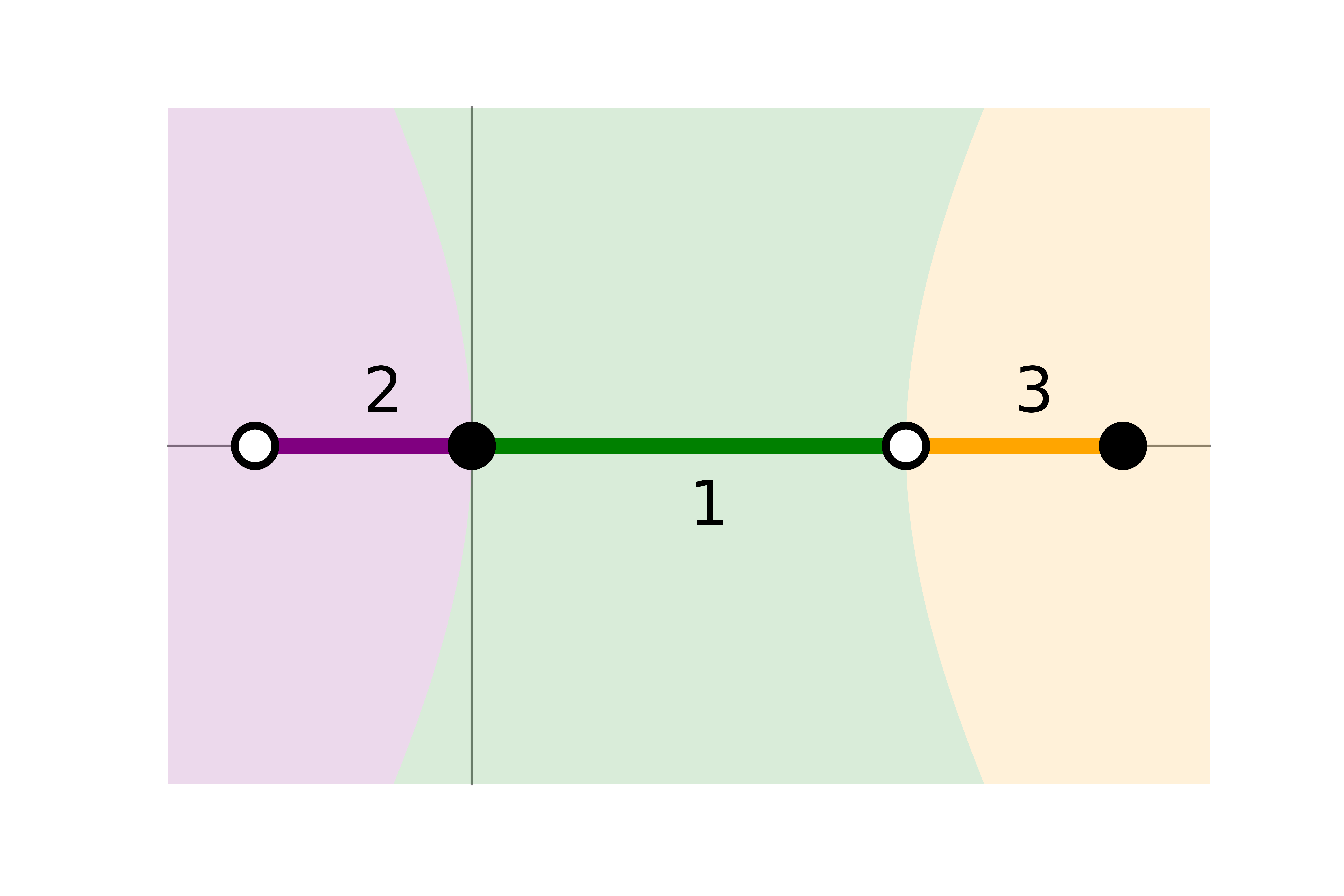}}\\\medskip%
    \subcaptionbox{The Bely\u{\i} map \(\beta \circ \gamma\); \textbf{\textsf{x}}'s indicate preimages of \(1/2\)}{%
        \includegraphics[width=0.75\textwidth]{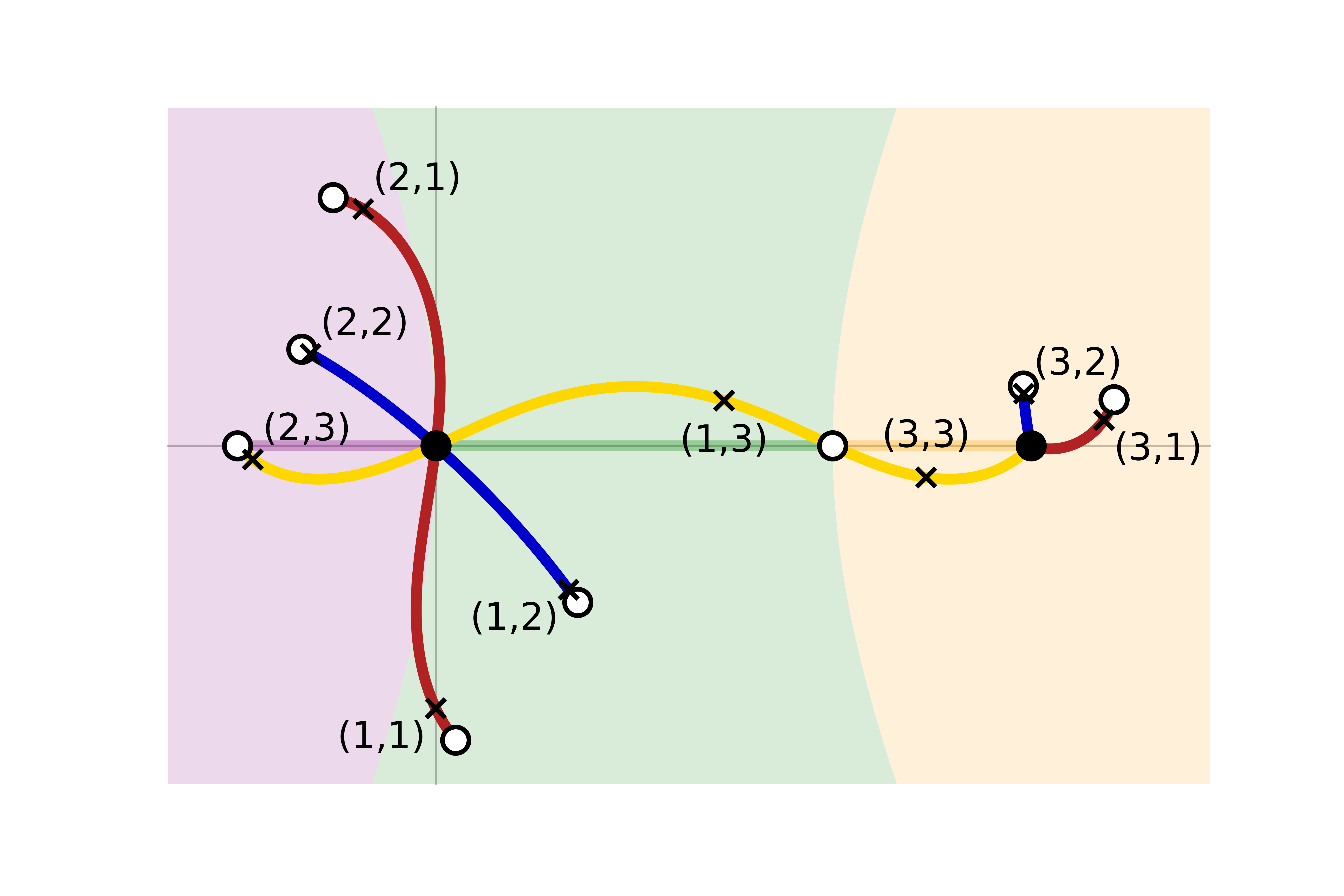}}%
\caption{A composition of Bely\u{\i} maps demonstrating the bijection \(E_{\beta\gamma} \leftrightarrow E_{\gamma} \times E_{\beta}\)}
\label{bijectionfig}
\end{figure}

\subsection{Action of Paths on Sheets}\hspace*{\fill}
\label{pathsonsheets}

As the edges of \(\beta \circ \gamma\) are liftings of \(E_{\beta}\) by \(\gamma\), a path \(p: [0, 1] \rightarrow Y\) between edges of \(\beta\) will lift to \(\deg \gamma\) paths between edges of \(\beta \circ \gamma\).  It is clear how these paths act on the \(E_{\beta}\) component of \(E_{\gamma} \times E_{\beta}\), but the question remains as to how they act on the \(E_{\gamma}\) component.  Analyzing the action on the sheets of \(\gamma\) by a path traversing between edges of \(\beta\) will provide the answer using the bijection between \(E_{\gamma}\) and the sheets of \(\gamma\).  The action is summarized nicely by Wood in \cite[Remark 3.6]{wood}, paraphrased here:

\hfill\begin{minipage}{\dimexpr\textwidth-3\parindent}
The path \(p\) can be viewed as an element of \(\pi_{1}(\mathbb{P}^{1}(\mathbb{C})_{*}, \mathcal{S})\), where the base-point is the sheet \(\mathcal{S}\).  However, \(\mathcal{S}\) can be canonically identified with the interval \((0, 1)\), hence the point \(1/2\), since \((0, 1)\) lies inside \(\mathcal{S}\).  This allows us to identify the path \(p\) with an element \(p^{\circlearrowleft} \in \pi_{1}(\mathbb{P}^{1}(\mathbb{C})_{*}, 1/2)\).  Then \(\sigma_{p^{\circlearrowleft}}\) is just the image of \(p^{\circlearrowleft}\) in \(\Mon \gamma\). 
\end{minipage}\newline

\noindent What is being described is that just as the monodromy of a Bely\u{\i} map \(\gamma\) is determined according to how the sets of points comprising the edges \(E_{\gamma}\) are mapped onto one another by the lifting a loop in \(\pi_{1}(\mathbb{P}^{1}(\mathbb{C})_{*}, 1/2)\), so too the action of a path \(p\) on the \(E_{\gamma}\) component of \(E_{\beta \circ \gamma}\) can be determined according to how the sets of points comprising the sheets of of \(\gamma\) are mapped onto one another. 

It is instructive to make this correspondence more explicit.

\begin{const}
Let \(p\) be a path in \(\mathbb{P}^{1}(\mathbb{C})_{*}\).  For \(j = 0, 1\), let \(s_{j} \coloneqq -\sgn(\re p(j))\) and let
\[\alpha_{p,j}(t) \coloneqq \begin{cases}
    \frac{1/2 + p(j)}{2} + \frac{1/2 - p(j)}{2}e^{s_{j}\pi it} & \textrm{if }p(j) \in (-\infty, 0) \cup (1, \infty),\\
    (1 - t)p(j) + \frac{1}{2}t & \textrm{else},
    \end{cases}\]
Finally, let
\[p^{\circlearrowleft} \coloneqq \alpha_{p,0}^{-1} \ast p \ast \alpha_{p,1}.\]
\label{loopconstruction}
\end{const}
\noindent If \(p(j)\), \(j = 0, 1\), lies in either interval \((-\infty, 0)\) or \((1, \infty)\), \(\alpha_{j}\) is a circular arc from \(p(j)\) through \(\mathbb{H}^{+}\) to \(\frac{1}{2}\).  Otherwise, it is the straight-line path from \(p(j)\) to \(\frac{1}{2}\).  The result of Construction \ref{loopconstruction} is a homomorphism of groupoids.

\begin{lem}
The function
\begin{align*}
\cdot^{\circlearrowleft}: \textrm{paths in }\mathbb{P}^{1}(\mathbb{C})_{*}\big/\!\simeq_{p} &\twoheadrightarrow \pi_{1}(\mathbb{P}^{1}(\mathbb{C})_{*}, 1/2),\\
p &\mapsto p^{\circlearrowleft},
\end{align*}
where \(\simeq_{p}\) is equivalence under path homotopy, is a surjective homomorphism of groupoids, and the action of \(p\) on the sheets of a Bely\u{\i} map \(\gamma\) is identical to the action of \(p^{\circlearrowleft}\).
\label{groupoidhom}
\end{lem}

\begin{proof}
First, if \(H\!\!: p_{1} \simeq_{p} p_{2}\), then because \(\alpha_{p_{j},i}\), \(j = 0, 1\), depends only on \(p_{j}(i)\), \(\id \ast \,H\! \ast \id\) is a path homotopy between \(p_{1}^{\circlearrowleft}\) and \(p_{2}^{\circlearrowleft}\), so that \(\cdot^{\circlearrowleft}\) is well-defined.  Let \(p_{1}, p_{2}\) be paths in \(\mathbb{P}^{1}(\mathbb{C})_{*}\) with \(p_{1}(1) = p_{2}(0)\).  Again because \(\alpha_{p_{j},i}\) depends only on \(p_{j}(i)\), \(\alpha_{p_{1},1} = \alpha_{p_{2},0}\) and \((p_{1}p_{2})^{\circlearrowleft} \simeq_{p} p_{1}^{\circlearrowleft}p_{2}^{\circlearrowleft}\).  Finally, because \(\alpha_{p_{j},i}\) does not cross \(\mathbb{P}^{1}(\mathbb{R})\), \(p^{\circlearrowleft}(i)\) lies in the same sheet as \(p(i)\) and the action of \(p^{\circlearrowleft}\) is the same as \(p\).
\end{proof}

As \(p^{\circlearrowleft} \in \pi_{1}(\mathbb{P}^{1}(\mathbb{C})_{*})\), it has an image \(\sigma_{p^{\circlearrowleft}}\) under the monodromy representation of \(\gamma\).  By the preceding lemma, the action of complicated paths can be constructed from the action of simpler paths.  To this end, it is useful to consider the following regions of \(\mathbb{P}^{1}(\mathbb{C})_{*}\), containing certain segments of \(\mathbb{P}^{1}(\mathbb{R})\), in order to describe crossing of \(\mathbb{P}^{1}(\mathbb{R})\):
\begin{equation*}
\begin{gathered}
    \mathcal{R}_{-1/2} \coloneqq \mathbb{P}^{1}(\mathbb{C})\backslash[0, \infty], \quad
    \mathcal{R}_{1/2} \coloneqq \mathbb{P}^{1}(\mathbb{C})\backslash([-\infty, 0] \cup [1, \infty]),\\
    {\centering \mathcal{R}_{3/2} \coloneqq \mathbb{P}^{1}(\mathbb{C})\backslash[-\infty, 1],}
\end{gathered}
\end{equation*}
where the subscript indicates the unique real point in the region lying on either \(e^{2\pi it}/2\) or \(1 - e^{2\pi it}/2\), \(t \in [0, 1]\).  More complicated paths can then be partitioned into finitely many paths, each lying in one of these regions, by the Heine-Borel theorem.  In particular, the following basic cases serve as building blocks in the determination of \(\sigma_{p^{\circlearrowleft}}\) for more complex paths \(p\) in \(\mathbb{P}^{1}(\mathbb{C})_{*}\), with an example of each case being illustrated in Figure \ref{pathloop}.

\begin{desc}\hspace{0pt}\vspace{-\baselineskip}\newline
\begin{enumerate}
    \item If \(p \subseteq \mathscr{R}_{1/2}\), then \(p^{{\circlearrowleft}} \simeq_{p} 1\).
    \item If either \(p(0), p(1) \in \overline{\mathbb{H}^{+}}\) or \(p(0), p(1) \in \mathbb{H}^{-}\) and either \(p \subseteq \mathscr{R}_{-1/2}\) or \(p \subseteq \mathscr{R}_{3/2}\), then \(p^{{\circlearrowleft}} \simeq_{p} 1\).
    \item If \(p(0) \in \overline{\mathbb{H}^{+}}\), \(p(1) \in \mathbb{H}^{-}\), and \(p \subseteq \mathscr{R}_{-1/2}\), then \(p^{{\circlearrowleft}} \simeq_{p} e^{2\pi it}/2\).\label{upperlowerinf0}
    \item If \(p(0) \in \overline{\mathbb{H}^{+}}\), \(p(1) \in \mathbb{H}^{-}\), and \(p \subseteq \mathscr{R}_{3/2}\), then \(p^{{\circlearrowleft}} \simeq_{p} 1 - e^{-2\pi it}/2\).\label{upperlower1inf}
    \item If \(p(0) \in \mathbb{H}^{-}\), \(p(1) \in \overline{\mathbb{H}^{+}}\), and \(p \subseteq \mathscr{R}_{-1/2}\), then \(p^{{\circlearrowleft}} \simeq_{p} e^{-2\pi it}/2\).\label{lowerupperinf0}
    \item If \(p(0) \in \mathbb{H}^{-}\), \(p(1) \in \overline{\mathbb{H}^{+}}\), and \(p \subseteq \mathscr{R}_{3/2}\), then \(p^{{\circlearrowleft}} \simeq_{p} 1 - e^{2\pi it}/2\).\label{lowerupper1inf}
\end{enumerate}
\label{casebreakdown}
\end{desc}

\stepcounter{thm}
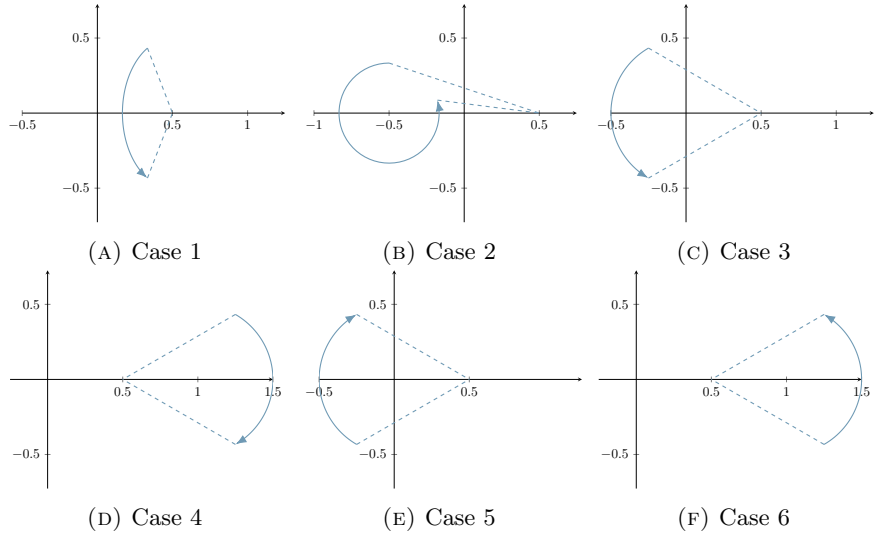
\begin{figure}[ht]
\centering
\begin{subfigure}{0.3\linewidth}
\centering
\scalebox{0.51}{\begin{tikzpicture}
\begin{axis}[xmin=-1/2, xmax=5/4,
ymin=-3/5, ymax=3/5, trig format plots=rad, axis equal, axis lines=middle]
\addplot [domain=1/3:2/3, samples=50, thick, slayter, -{Latex[scale=1.25]}] ({1/2+cos(2*pi*x)/3}, {sin(2*pi*x)/2});
\addplot [domain=0:1, samples=50, thick, dashed, slayter] ({(1-x)*(1/2+cos(2*pi*1/3)/3)+x*1/2}, {(1-x)*sin(2*pi*1/3)/2});
\addplot [domain=0:1, samples=50, thick, dashed, slayter] ({(1-x)*(1/2+cos(2*pi*2/3)/3)+x*1/2}, {(1-x)*sin(2*pi*2/3)/2});
\end{axis}
\end{tikzpicture}}
\caption{Case \(1\)}
\end{subfigure}
\begin{subfigure}{0.3\linewidth}
\centering
\scalebox{0.51}{\begin{tikzpicture}
\begin{axis}[xmin=-1, xmax=3/4,
ymin=-3/5, ymax=3/5, trig format plots=rad, axis equal, axis lines=middle]
\addplot [domain=1/4:25/24, samples=50, thick, slayter, -{Latex[scale=1.25]}] ({-1/2+cos(2*pi*x)/3}, {sin(2*pi*x)/3});
\addplot [domain=0:1, samples=50, thick, dashed, slayter] ({(1-x)*(cos(2*pi*1/4)/3-1/2)+x*1/2}, {(1-x)*sin(2*pi*1/4)/3});
\addplot [domain=0:1, samples=50, thick, dashed, slayter] ({(1-x)*(cos(2*pi*25/24)/3-1/2)+x*1/2}, {(1-x)*sin(2*pi*25/24)/3});
\end{axis}
\end{tikzpicture}}
\caption{Case \(2\)}
\end{subfigure}
\begin{subfigure}{0.3\linewidth}
\centering
\scalebox{0.51}{\begin{tikzpicture}
\begin{axis}[xmin=-1/2, xmax=5/4,
ymin=-3/5, ymax=3/5, trig format plots=rad, axis equal, axis lines=middle]
\addplot [domain=1/3:2/3, samples=50, thick, slayter, -{Latex[scale=1.25]}] ({cos(2*pi*x)/2}, {sin(2*pi*x)/2});
\addplot [domain=0:1, samples=50, thick, dashed, slayter] ({(1-x)*cos(2*pi*1/3)/2+x*1/2}, {(1-x)*sin(2*pi*1/3)/2});
\addplot [domain=0:1, samples=50, thick, dashed, slayter] ({(1-x)*cos(2*pi*2/3)/2+x*1/2}, {(1-x)*sin(2*pi*2/3)/2});
\end{axis}
\end{tikzpicture}}
\caption{Case \(3\)}
\end{subfigure}
\begin{subfigure}[b]{0.3\linewidth}
\centering
\scalebox{0.51}{
\begin{tikzpicture}
\begin{axis}[xmin=-1/4, xmax=3/2,
ymin=-3/5, ymax=3/5, trig format plots=rad, axis equal, axis lines=middle]
\addplot [domain=1/3:2/3, samples=50, thick, slayter, -{Latex[scale=1.25]}] ({1-cos(2*pi*x)/2}, {sin(2*pi*x)/2});
\addplot [domain=0:1, samples=50, thick, dashed, slayter] ({(1-x)*(1-cos(2*pi*1/3)/2)+x*1/2}, {(1-x)*sin(2*pi*1/3)/2});
\addplot [domain=0:1, samples=50, thick, dashed, slayter] ({(1-x)*(1-cos(2*pi*2/3)/2)+x*1/2}, {(1-x)*sin(2*pi*2/3)/2});
\end{axis}
\end{tikzpicture}}
\caption{Case \(4\)}
\end{subfigure}
\begin{subfigure}[b]{0.3\linewidth}
\centering
\scalebox{0.51}{
\begin{tikzpicture}
\begin{axis}[xmin=-1/2, xmax=5/4,
ymin=-3/5, ymax=3/5, xtick={-0.5, 0.5}, ytick={-0.5, 0.5}, trig format plots=rad, axis equal, axis lines=middle]
\addplot [domain=1/3:2/3, samples=50, thick, slayter, -{Latex[scale=1.25]}] ({cos(-2*pi*x)/2}, {sin(-2*pi*x)/2});
\addplot [domain=0:1, samples=50, thick, dashed, slayter] ({(1-x)*cos(2*pi*1/3)/2+x*1/2}, {(1-x)*sin(2*pi*1/3)/2});
\addplot [domain=0:1, samples=50, thick, dashed, slayter] ({(1-x)*cos(2*pi*2/3)/2+x*1/2}, {(1-x)*sin(2*pi*2/3)/2});
\end{axis}
\end{tikzpicture}}
\caption{Case \(5\)}
\end{subfigure}
\begin{subfigure}[b]{0.3\linewidth}
\centering
\scalebox{0.51}{
\begin{tikzpicture}
\begin{axis}[xmin=-1/4, xmax=3/2,
ymin=-3/5, ymax=3/5, trig format plots=rad, axis equal, axis lines=middle]
\addplot [domain=1/3:2/3, samples=50, thick, slayter, -{Latex[scale=1.25]}] ({1-cos(-2*pi*x)/2}, {sin(-2*pi*x)/2});
\addplot [domain=0:1, samples=50, thick, dashed, slayter] ({(1-x)*(1-cos(2*pi*1/3)/2)+x*1/2}, {(1-x)*sin(2*pi*1/3)/2});
\addplot [domain=0:1, samples=50, thick, dashed, slayter] ({(1-x)*(1-cos(2*pi*2/3)/2)+x*1/2}, {(1-x)*sin(2*pi*2/3)/2});
\end{axis}
\end{tikzpicture}}
\caption{Case \(6\)}
\end{subfigure}
\caption{The path-homotopy classes of extensions of basic paths}
\label{pathloop}
\end{figure}

\subsection{Action of Loops on Edges}\hspace*{\fill}

Now that the identification of a path \(p \subseteq \mathbb{P}^{1}(\mathbb{C})_{*}\) with an element \(p^{\circlearrowleft} \in \pi_{1}(\mathbb{P}^{1}(\mathbb{C})_{*})\) has been made, it is possible to determine, for a given \(\beta\), how a loop \(\lambda \in \pi_{1}(\mathbb{P}^{1}(\mathbb{C})_{*})\) would permute the edges of an arbitrary Bely\u{\i} map \(\beta \circ \gamma\).  At this point, the extending pattern of a Bely\u{\i} map \(\beta\) (cf.\ \cite[Section 3.2]{wood}) is introduced in order to describe the action of a loop \(\lambda\) on an edge \((e_{\gamma}, e_{\beta})\).

\begin{defnnum}
Let \(\beta: Y \rightarrow Z\) be a dynamical Bely\u{\i} map.
\begin{enumerate}
    \item The extending pattern of \(\beta\) from \(\lambda \in \pi_{1}(Z_{*}, 1/2)\) is the function \(f_{\lambda}: E_{\beta} \rightarrow \pi_{1}(Y_{*}, 1/2)\) defined as follows.  For \(e_{\beta} \in E_{\beta}\), lift \(\lambda\) by \(\beta\) to \(\lambda_{Y}\) with \(\lambda_{Y}(0) = e_{\beta}(1/2)\) and form \(\lambda_{Y}^{\circlearrowleft}\) following Construction \ref{loopconstruction}.  Then
\[f_{\lambda}(e_{\beta}) \coloneqq \lambda_{Y}^{\circlearrowleft} \in \pi_{1}(Y_{*}, 1/2).\]
    \item The extending pattern of \(\beta\) is the pair \((f_{0}, f_{1}) \coloneqq (f_{\lambda_{0}}, f_{\lambda_{1}})\), where \(\lambda_{0}\), respectively \(\lambda_{1}\), is a loop with winding number one around \(0\), respectively \(1\), and winding number zero around both \(1\) and \(\infty\), respectively \(0\) and \(\infty\).
\end{enumerate}
\label{extenddefn}
\end{defnnum}

\stepcounter{thm}
\begin{figure}[ht]
	\centering
        \includegraphics[scale=0.5]{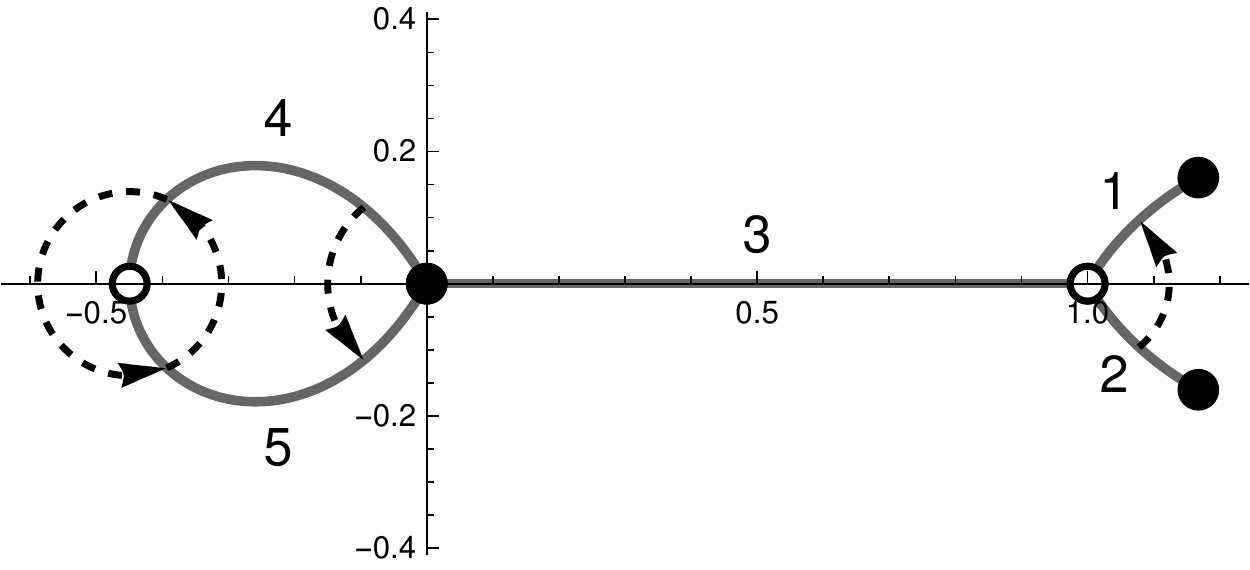}
        \caption{The extending pattern of a Bely\u{\i} map}
        \label{paths}
\end{figure}

\bigskip
\begin{example}
\label{extendingexample}
Define \(\lambda_{0} \coloneqq e^{2\pi i t}/2\) and \(\lambda_{1} \coloneqq 1 - e^{2\pi i t}/2\), and let their homotopy classes be denoted by \(a\) and \(b\), respectively.  An example of determining the extending pattern is shown in Figure \ref{paths}.  The only lifting of \(\lambda_{0}\) which crosses \(\mathbb{P}^{1}(\mathbb{R})\) outside of \((0, 1)\) is the lifting beginning at Edge \(4\).  As the lifting crosses \((-\infty, 0)\) from \(\overline{\mathbb{H}^{+}}\) to \(\mathbb{H}^{-}\), Description \ref{casebreakdown} prescribes that \(f_{0}(4) = a\), while \(f_{0}(i) = 1\) for \(i \not= 4\).  On the other hand, liftings of \(\lambda_{1}\) beginning at Edge \(4\) and Edge \(5\) cross \((-\infty, 0)\) in opposite directions, so that \(f_{1}(4) = a\) and \(f_{1}(5) = a^{-1}\).  Further, a lifting of \(\lambda_{1}\) beginning at Edge \(2\) crosses \((1, \infty)\) from \(\mathbb{H}^{-}\) to \(\overline{\mathbb{H}^{+}}\), hence \(f_{1}(2) = b\).  Finally, \(f_{1}(1) = f_{1}(3) = 1\).
\end{example}

\bigskip
\begin{example}
For \(\beta\) defined in \eqref{betadef} and shown in Figure \ref{bijectionfig}(A), the extending pattern is \(f_{0} = (1, 1, b), f_{1} = (a, 1, 1)\) because the path from Edge \(3\) to Edge \(1\) crosses \((1, \infty)\) from \(\mathbb{H}^{-}\) to \(\overline{\mathbb{H}^{+}}\), while the loop around the white vertex of Edge \(1\) crosses \((-\infty, 0)\) from \(\overline{\mathbb{H}^{+}}\) to \(\mathbb{H}^{-}\).
\label{extendingexample2}
\end{example}

\begin{thm}
Let \(\beta: Y \rightarrow Z\) be a dynamical Bely\u{\i} map.  For any Bely\u{\i} map \(\gamma: X \rightarrow Y\), the action of the loop \(\lambda \in \pi_{1}(Z_{*}, 1/2)\) on the edge \((e_{\gamma}, e_{\beta}) \in E_{\beta \circ \gamma}\) is given by
\begin{equation}
(e_{\gamma}, e_{\beta})^{\lambda} = \Big(e_{\gamma}^{f_{\lambda}(e_{\beta})}, e_{\beta}^{\lambda}\Big),
\label{extendedaction}
\end{equation}
where each action is the monodromy action of the respective Bely\u{\i} map.
\label{actiononbetagamma}
\end{thm}

\begin{proof}
Let \(\lambda \in \pi_{1}(Z_{*}, 1/2)\), let \(e = (e_{\gamma}, e_{\beta}) \in E_{\beta \circ \gamma}\), let \(\lambda_{Y}\) be as in the statement of the theorem, and let \(\lambda_{Y}^{\circlearrowleft} = f_{\lambda}(e_{\beta})\).  Further, let \(\lambda_{X} \subseteq X\) be the lifting of \(\lambda\) by \(\beta \circ \gamma\) with \(\lambda_{X}(0) = e(1/2)\).  From Proposition \ref{edgebijection}, to show that the first component of \(e^{\lambda}\) is \(e_{\gamma}^{\lambda_{Y}^{\circlearrowleft}}\), it is necessary to show that \(e_{\gamma}^{\lambda_{Y}^{\circlearrowleft}}\) is the edge of \(\gamma\) lying in the sheet containing \(e^{\lambda}(1/2) = \lambda_{X}(1)\).

By Proposition \ref{edgebijection}, \(\gamma(e) = e_{\beta}\), so that \(\gamma\big(\lambda_{X}(0)\big) = e_{\beta}(1/2) = \lambda_{Y}(0)\).  Moreover, \(\beta(\gamma \circ \lambda_{X}) = \lambda\) shows that \(\gamma \circ \lambda_{X}\) is a lifting of \(\lambda\) by \(\beta\) and it follows that \(\gamma(\lambda_{X}) = \lambda_{Y}\).  Then \(\lambda_{Y}\) acts on the sheets of \(\gamma\) by
\[\lambda_{Y}: \begin{tabular}{c}
\textrm{sheet of \(\gamma\)}\\
\textrm{containing \(\lambda_{X}(0)\)}
\end{tabular} \longmapsto \begin{tabular}{c}
\textrm{sheet of \(\gamma\)}\\
\textrm{containing \(\lambda_{X}(1)\)}
\end{tabular}.\]
By Lemma \ref{groupoidhom}, \(\lambda_{Y}^{\circlearrowleft}\) permutes the sheets, hence edges, of \(\gamma\) in the same way and \(e_{\gamma}^{\lambda_{Y}^{\circlearrowleft}}\) lies in the sheet containing \(\lambda_{X}(1)\).

Finally, because \(\gamma(\lambda_{X}) = \lambda_{Y}\),
\[\gamma\left(e^{\lambda}\bigg(\frac{1}{2}\bigg)\right) = \gamma\bigg(\lambda_{X}(1)\bigg) = \lambda_{Y}(1) = e_{\beta}^{\lambda}\Big(\frac{1}{2}\Big),\]
and the second component of \(e^{\lambda}\) is given by \(e_{\beta}^{\lambda}\), completing the result.
\end{proof}

\stepcounter{thm}
\begin{figure}[ht]
    \centering
    \subcaptionbox{A loop \(\lambda\) around \(0\) on the Bely\u{\i} map \(f(x) = x\)}{%
        \includegraphics[width=0.475\textwidth]{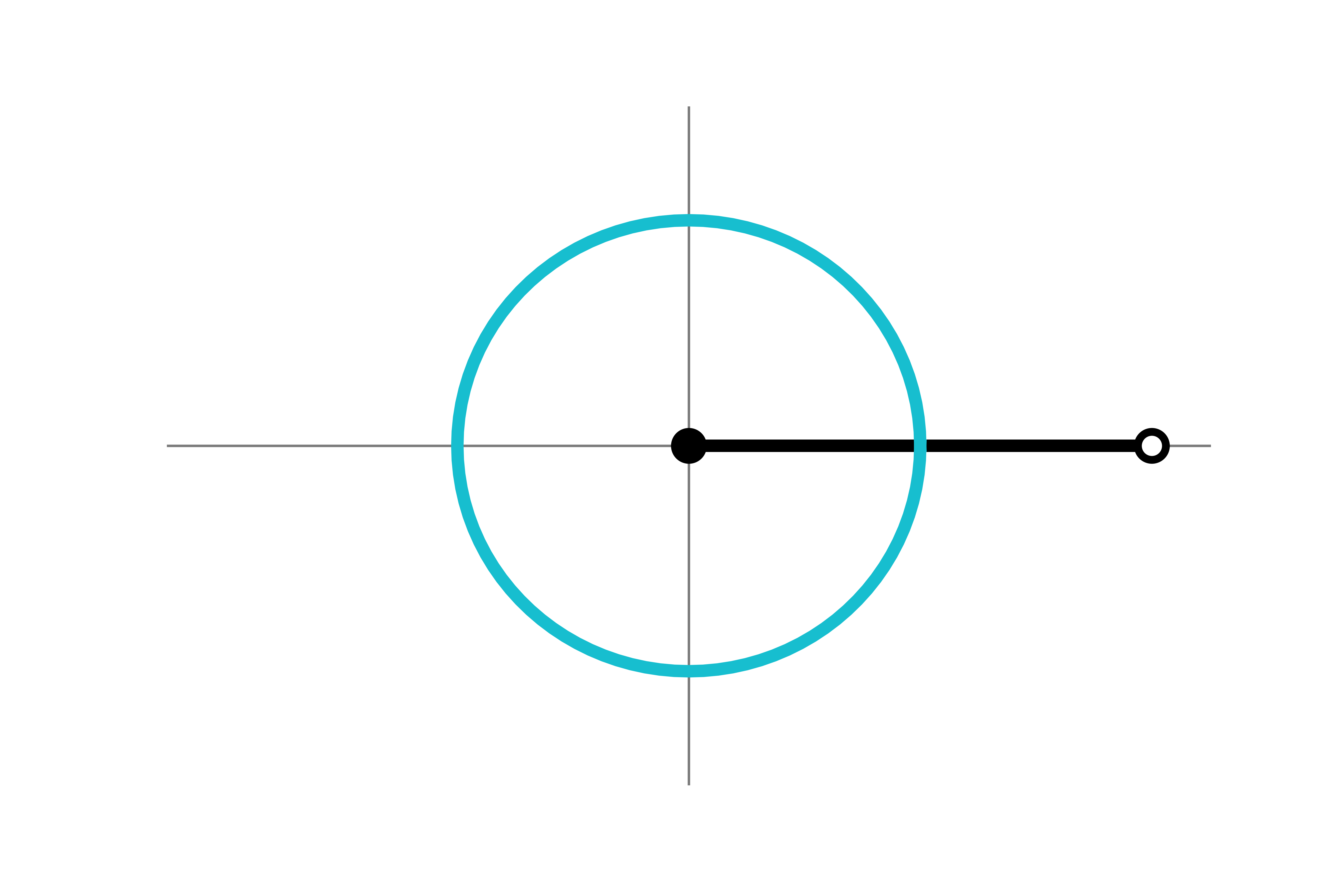}}\hspace{0.01\textwidth}%
    \subcaptionbox{A lifting \(\lambda_{Y}\) of \(\lambda\) and the creation of a loop \(\lambda_{Y}^{\circlearrowleft}\)}{%
    \includegraphics[width=0.475\textwidth]{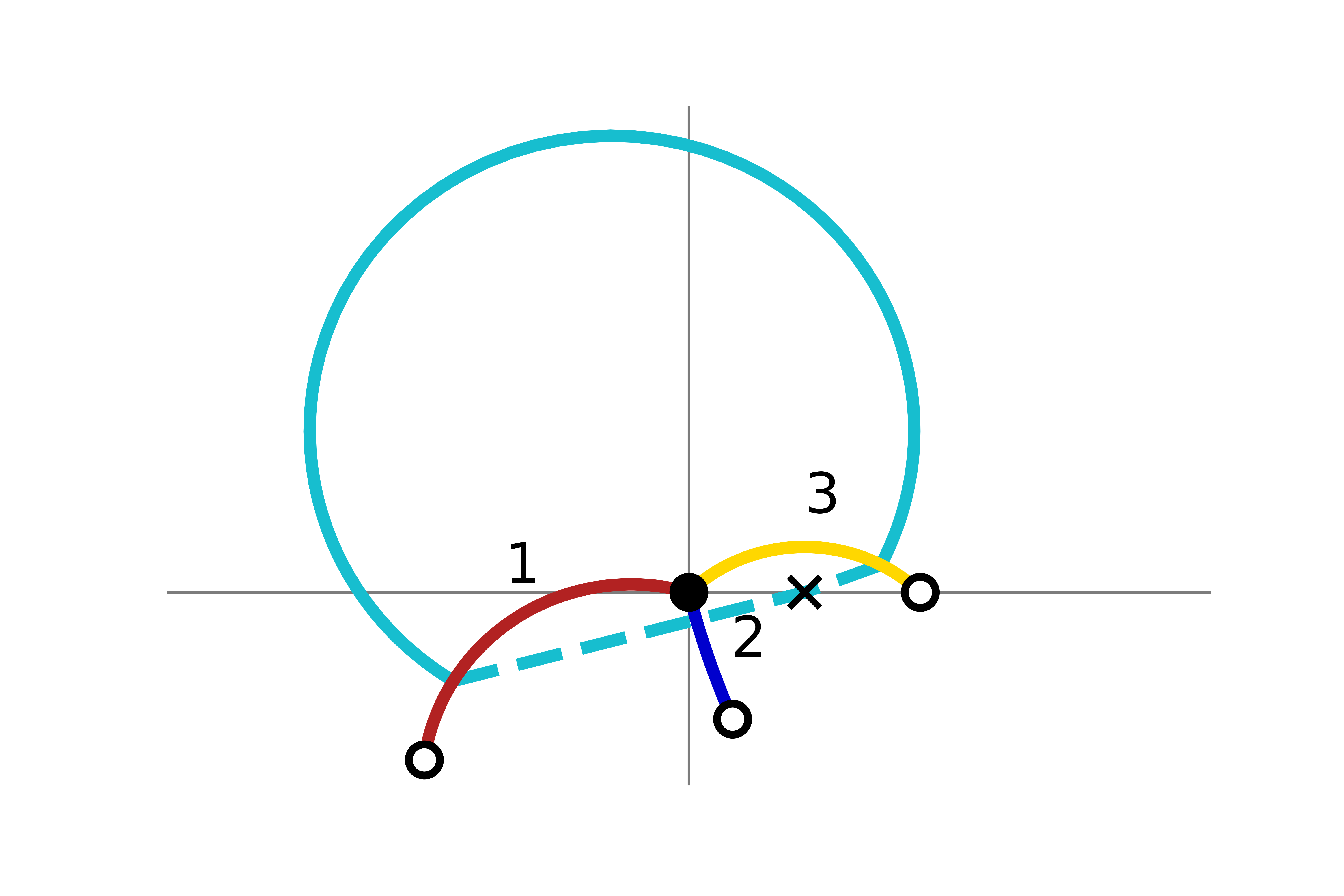}}\\\medskip%
    \subcaptionbox{A lifting \(\lambda_{X}\) of \(\lambda_{Y}\) and a lifting of \(\lambda_{Y}^{\circlearrowleft}\)}{%
    \includegraphics[width=0.75\textwidth]{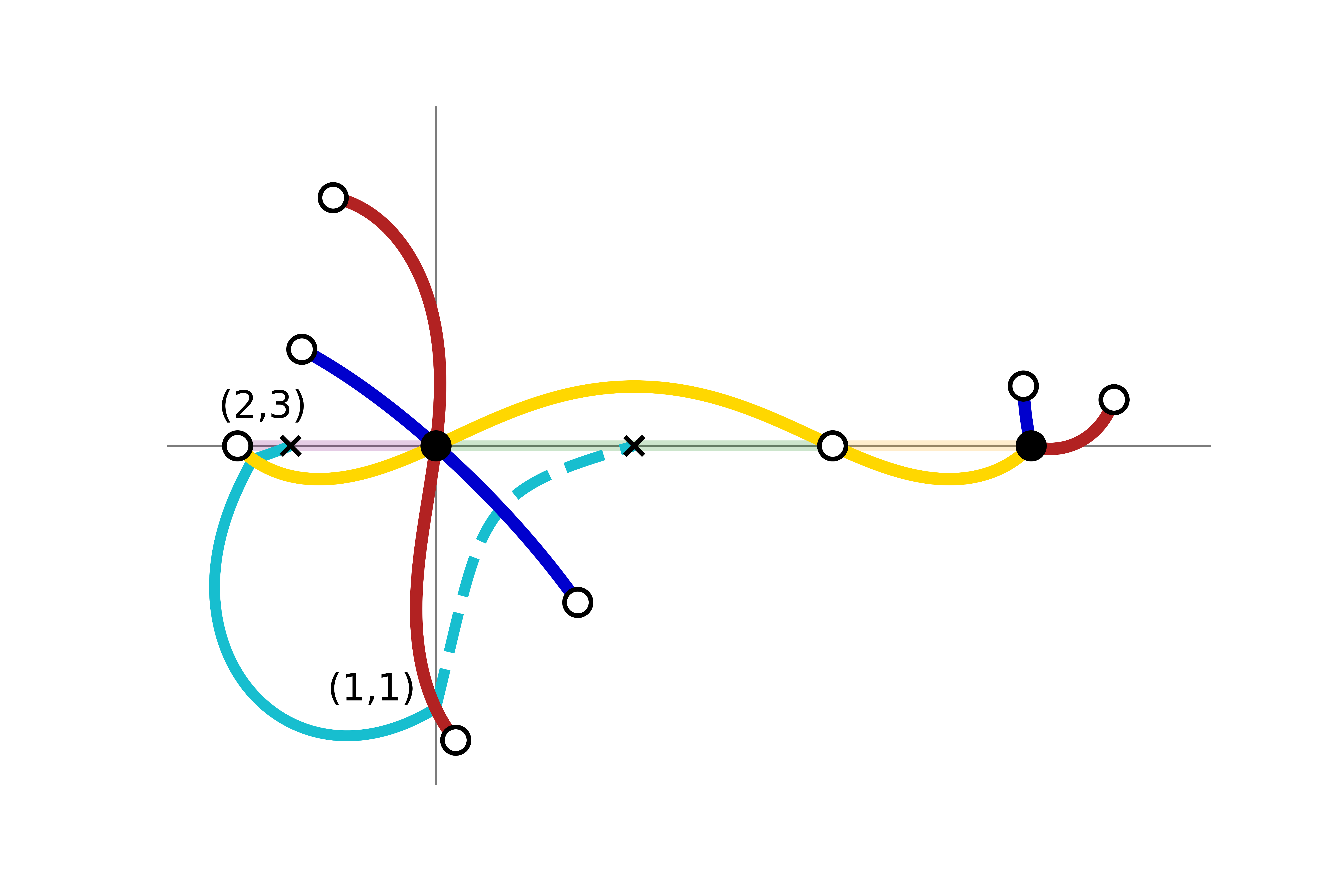}}%
\caption{Lifting a loop in \(Z\) by \(\beta \circ \gamma\) determines a path between edges of \(\gamma\)}
    \label{action}
\end{figure}

\begin{example}
Theorem \ref{actiononbetagamma} is illustrated in Figure \ref{action}.  The loop \(\lambda \subset Z\) around \(1\) is lifted to \(\lambda_{Y} \subset Y\), spanning between Edge \(2\) and Edge \(1\) of \(\beta\).  Next, \(\lambda_{Y}\) is used to form \(\lambda_{Y}^{\circlearrowleft}\) according to Construction \ref{loopconstruction}, as represented by the dashed segments.   Being a loop in \(Y\) based at \(1/2\), \(\lambda_{Y}^{\circlearrowleft}\), which includes \(\lambda_{Y}\) as a subset, is lifted by \(\gamma\) to a path between edges of \(\gamma\), which simultaneously lifts \(\lambda_{Y}\) to a path \(\lambda_{X} \subset X\) between edges of \(\beta \circ \gamma\).  The path \(\lambda_{X}\) demonstrates the action of \(\lambda_{Y}\) on \(E_{\beta \circ \gamma}\), while the lifting of \(\lambda_{Y}^{\circlearrowleft}\) creates a correspondence between the endpoints of \(\lambda_{X}\) and edges of \(\gamma\).  In the notation of Theorem \ref{actiononbetagamma}, let \(\lambda_{1} \subset Z\) be a loop around \(1\) and let \(b \subset Y\) also be a loop around \(1\).  Since \(f_{\lambda_{1}} = (1, 1, b)\), where \(b\) acts on the edges of \(\gamma\) by \((1\;3) \in \Mon \gamma\),
\[(1, 3)^{\lambda_{1}} = (1^{f_{\lambda_{1}}(3)}, 3^{\lambda_{1}}) = (1^{b}, 1) = (3, 1).\]
\end{example}

\bigskip
From \eqref{extendedaction}, it becomes clear that the action of a loop \(\lambda\) on an edge \(e \in E_{\beta \circ \gamma}\) can be determined from the monodromy of \(\beta\), the monodromy of \(\gamma\), and \(f_{\lambda}\).  Perhaps the most interesting observation from Theorem \ref{actiononbetagamma} is that \(f_{\lambda}\) is independent of \(\gamma\).  Thus, the effect of composition with the Bely\u{\i} map \(\beta\) on the monodromy of any Bely\u{\i} map \(\gamma\) can be established once and for all by determining \(f_{\lambda}\) for loops \(\lambda_{0}\) and \(\lambda_{1}\) around \(0\) and \(1\), respectively, then specializing to the function \(\gamma\) at hand.

\section{Monodromy as a Wreath Product}
\label{monwreath}

Theorem \ref{actiononbetagamma} enables the development of the group structure of \(\Mon \beta \circ \gamma\) as a wreath product.\ \cite[Prop. 1.7.10]{landozvonkin}  The fact that \(\Mon \beta \circ \gamma\) is a subgroup of \(\Mon \gamma \wr_{E_{\beta}} \Mon \beta\) begins to appear in \eqref{extendedaction}, in which the action of \(\lambda\) on \(e_{\gamma}\) is a function of \(e_{\beta}\).  As the codomain of \(f_{\lambda}\) is \(\pi_{1}^{Z}\), rather than \(\Mon \gamma\), \(\pi_{1}^{Z} \wr_{E_{\beta}} \Mon \beta\) will be considered first, reflecting the independence from \(\gamma\) of the effect of composition with \(\beta\) on the monodromy action.  Then \(\pi_{1}^{Z} \wr_{E_{\beta}} \Mon \beta\) will be mapped onto \(\Mon \gamma \wr_{E_{\beta}} \Mon \beta\).

\subsection{The Extended Monodromy Group}\hspace*{\fill}

\begin{thm}
    Let \(\beta: Y \rightarrow Z\) be a dynamical Bely\u{\i} map, let
    \[\pi_{1}^{Y} \coloneqq \pi_{1}(Y_{*}, 1/2), \qquad \pi_{1}^{Z} \coloneqq \pi_{1}(Z_{*}, 1/2),\]
    let \(\rho_{\beta}\) be the monodromy representation of \(\beta\), let \(\tau_{\lambda} \coloneqq \rho_{\beta}(\lambda)\), and let \(f_{\lambda}\) be defined as in Definition \ref{extenddefn}.  Define the action of \(\Mon \beta\) on \(\Fun(E_{\beta}, \pi_{1}^{Y})\) by \(f^{\tau}(e_{\beta}) = f(e_{\beta}^{\tau^{-1}})\).  For \(\lambda \in \pi_{1}^{Z}\), define
    \[\varphi_{\beta}(\lambda) \coloneqq f_{\lambda} \rtimes \tau_{\lambda}.\]
    Then \(\varphi_{\beta}\) is a homomorphism
    \[\varphi_{\beta}: \pi_{1}^{Z} \rightarrow \pi_{1}^{Y} \wr_{E_{\beta}} \Mon \beta = \Fun(E_{\beta}, \pi_{1}^{Y}) \rtimes \Mon \beta\]
    whose image is an extension by \(\Mon \beta\) of
    \[\ker \Big(\proj_{\beta}\raisebox{-0.5ex}{\big|}_{\varphi_{\beta}(\pi_{1}^{Z})}\Big) = \varphi_{\beta}(\ker \rho_{\beta}) \approx \ker \rho_{\beta}/\ker \varphi_{\beta},\]
    where \(\proj_{\beta}: \varphi_{\beta}(\pi_{1}^{Z}) \twoheadrightarrow \Mon_{\beta}\) is restricted to \(\varphi_{\beta}(\pi_{1}^{Z})\).
\label{extendedmon}
\end{thm}

\begin{proof}
The situation in question is illustrated by
\begin{center}
\begin{tikzcd}
    \pi_{1}^{Z}\arrow{r}{\varphi_{\beta}}\arrow[swap]{rd}{\rho_{\beta}} & \pi_{1}^{Y} \wr_{E_{\beta}} \Mon \beta \arrow{d}{\proj_{\beta}}\\
    & \Mon \beta
\end{tikzcd}.
\end{center}
To begin, because \(\rho_{\beta}\) and \(\cdot^{\circlearrowleft}\) are well-defined, \(\varphi_{\beta}\) is well-defined.  Now, let \(\lambda_{1}, \lambda_{2} \in \pi_{1}^{Z}\).  To see that \(\varphi_{\beta}\) is a homomorphism, first note that because \(\rho_{\beta}\) is a homomorphism, \(\tau_{\lambda_{1}}\tau_{\lambda_{2}} = \tau_{\lambda_{1} \ast \lambda_{2}}\) and
\[\varphi_{\beta}(\lambda_{1}) \varphi_{\beta}(\lambda_{2}) = (f_{\lambda_{1}}, \tau_{\lambda_{1}})(f_{\lambda_{2}}, \tau_{\lambda_{2}}) = (f_{\lambda_{1}} \cdot f_{\lambda_{2}}^{\tau_{\lambda_{1}}^{-1}}, \tau_{\lambda_{1} \ast \lambda_{2}}).\]
As \(\varphi_{\beta}(\lambda_{1} \ast \lambda_{2}) = (f_{\lambda_{1} \ast \lambda_{2}}, \tau_{\lambda_{1} \ast \lambda_{2}})\), it remains to show that \(f_{\lambda_{1} \ast \lambda_{2}} = f_{\lambda_{1}} \cdot f_{\lambda_{2}}^{\tau_{\lambda_{1}}^{-1}}\).

Let \(e_{\beta} \in E_{\beta}\) and lift \(\lambda_{1}\) by \(\beta\) to \(\tilde{\lambda}_{1}\) so that \(\tilde{\lambda}_{1}(0) \in e_{\beta}\).  Further, lift \(\lambda_{2}\) by \(\beta\) to \(\tilde{\lambda}_{2}\) so that \(\tilde{\lambda}_{2}(0) = \tilde{\lambda}_{1}(1)\).  The uniqueness of liftings guarantees that the \(\tilde{\lambda}_{1} \ast \tilde{\lambda}_{2}\) is identical to the lifting of \(\lambda_{1} \ast \lambda_{2}\) beginning at \(e_{\beta}(1/2)\).  By definition, \(e_{\beta}^{\tau_{\lambda_{1}}}\) is the edge containing \(\tilde{\lambda}_{1}(1) = \tilde{\lambda}_{2}(0)\), and
\begin{align*}
(f_{\lambda_{1}} \cdot f_{\lambda_{2}}^{\tau_{\lambda_{1}}^{-1}})(e_{\beta}) &= f_{\lambda_{1}}(e_{\beta}) \ast f_{\lambda_{2}}(e_{\beta}^{\tau_{\lambda_{1}}})\\
&= \tilde{\lambda}_{1}^{\circlearrowleft} \ast \tilde{\lambda}_{2}^{\circlearrowleft}\\
&\simeq_{p} (\tilde{\lambda}_{1} \ast \tilde{\lambda}_{2})^{\circlearrowleft} \qquad \textrm{(by Lemma \ref{groupoidhom})}\\
&= f_{\lambda_{1} \ast \lambda_{2}}(e_{\beta}).
\end{align*}
As \(e_{\beta}\) was arbitrary, this shows that \(\varphi_{\beta}\) is a homomorphism.

As
\[\lambda \in \ker \rho_{\beta} \Longleftrightarrow \lambda \in \ker(\proj_{\beta} \circ \varphi_{\beta}) \Longleftrightarrow \varphi_{\beta}(\lambda) \in \ker \proj_{\beta}\]
and \(\varphi_{\beta}(\ker \rho_{\beta}) \approx \ker \rho_{\beta}/\ker \varphi_{\beta}\) by the first isomorphism theorem, the result follows.
\end{proof}

The homomorphism \(\varphi_{\beta}\), and its image, serve to capture the important information regarding the effect that \(\beta\) has on the monodromy of a Bely\u{\i} map \(\gamma\) when composing the functions.

\begin{defnnum}
The extended monodromy representation of the dynamical Bely\u{\i} map \(\beta: Y \rightarrow Z\) is defined by
\begin{align*}
    \varphi_{\beta}: \pi_{1}^{Z} &\longrightarrow \pi_{1}^{Y} \wr_{E_{\beta}} \Mon \beta,\\
    \lambda &\longmapsto f_{\lambda} \rtimes \rho_{\beta}(\lambda),
\end{align*}
where \(f_{\lambda}\) is the extending pattern of \(\beta\) from \(\lambda\) and \(\rho_{\beta}\) is the monodromy representation of \(\beta\).  Further, the extended monodromy group of \(\beta\), \(\EMon \beta\), is the image of the extended monodromy representation, \(\varphi_{\beta}(\pi_{1}^{Z})\).
\end{defnnum}

\begin{cor}
The extended monodromy group of \(\beta\), \(\EMon \beta\), is given as the extension
\[1 \rightarrow \ker \rho_{\beta}/\ker \varphi_{\beta} \rightarrow \EMon \beta \rightarrow \Mon \beta \rightarrow 1,\]
where \(\rho_{\beta}\) and \(\varphi_{\beta}\) are the monodromy representation and extended monodromy representation of \(\beta\), respectively.
\label{projbeta}
\end{cor}

\begin{defnnum}
The monodromy extending group of \(\beta\), \(\MonExt \beta\), is the complement of \(\Mon \beta\) in \(\EMon \beta\) and is given by
\[\MonExt \beta = \ker \rho_{\beta}/\ker \varphi_{\beta}.\]
\end{defnnum}

\begin{example}
In computing the monodromy extending group of \(\beta(x)\) from \eqref{betadef}, writing \texttt{a},\texttt{b} for \(\lambda_{0}, \lambda_{1}\), \(\ker \rho_{\beta}\) is computed first as\newline
\noindent\verb|gap> MonBeta := Group((1,2,3), ());;|\newline
\noindent\verb|gap> rho := GroupHomomorphismByImages(FreeGroup("a", "b"), MonBeta);;|\newline
\noindent\verb|gap> GeneratorsOfGroup(Kernel(rho));|\newline
\noindent\verb|[ b, a^3, a*b*a^-1, a^-1*b*a ]|\newline
From Example \ref{extendingexample2}, \(f_{\lambda_{0}} = (1, 1, b)\) and \(f_{\lambda_{1}} = (a, 1, 1)\), so that
\[\varphi_{\beta}(\lambda_{0}) = (1, 1, b) \rtimes (1\;2\;3), \qquad \varphi_{\beta}(\lambda_{1}) = (a, 1, 1) \rtimes 1.\]
Then from \(\varphi_{\beta}(\lambda\ker \rho_{\beta}) \approx \ker \rho_{\beta}/\ker \varphi_{\beta}\), applying \(\varphi_{\beta}\) to the generators of \(\ker \rho_{\beta}\) above results in
\[\MonExt \beta \approx \big\langle (a, 1, 1) \rtimes 1,\: (b, b, b) \rtimes 1,\: (1, 1, a) \rtimes 1,\: (1, a, 1) \rtimes 1 \big\rangle.\]
\end{example}

\subsection{The Monodromy Group of a Composition}\hspace*{\fill}

Having established the extended monodromy group of \(\beta\) allows for determining \(\Mon \beta \circ \gamma\), for a Bely\u{\i} map \(\gamma\), simply through postcomposition of the extending patterns of \(\beta\) by the monodromy representation of \(\gamma\).

\begin{thm}
    Let \(\beta\) be a dynamical Bely\u{\i} map.  For any Bely\u{\i} map \(\gamma\), let \(\rho_{\gamma}\) be its monodromy representation and define
    \[\rho_{\gamma*}: \Fun(E_{\beta}, \pi_{1}^{Y}) \rightarrow \Fun(E_{\beta}, \Mon \gamma)\]
    by \(\rho_{\gamma*}(f) = \rho_{\gamma} \circ f\).  Then
    \[\varphi_{\gamma} \coloneqq (\rho_{\gamma*} \rtimes \id) \circ \varphi_{\beta}: \pi_{1}^{Z} \rightarrow \Mon \gamma \wr_{E_{\beta}} \Mon \beta\]
    coincides with \(\rho_{\beta\gamma}\).
\end{thm}

\begin{proof}
Showing that \(\ker \rho_{\beta \circ \gamma} = \ker \varphi_{\gamma}\) will imply that \(\Mon \beta \circ \gamma \approx \varphi_{\gamma}(\pi_{1}^{Z})\) by the first isomorphism theorem.

If \(\lambda \in \ker \rho_{\beta \circ \gamma}\), then by \eqref{extendedaction}, \(\rho_{\beta}(\lambda) = 1\) and \(\rho_{\gamma} \circ f_{\lambda} = 1\), so that \(\lambda \in \ker \varphi_{\gamma}\).  On the other hand, if \(\lambda \in \ker \varphi_{\gamma}\), then \(\rho_{\beta} = 1\) and \(\rho_{\gamma} \circ f_{\lambda} = 1\), and by \eqref{extendedaction}, \(\lambda \in \ker \rho_{\beta \circ \gamma}\).
\end{proof}

\begin{example}
For \(\gamma\) as in \eqref{betadef}, applying \(\rho_{\gamma*}\) to \(\MonExt \beta\) as computed in Example 4.5 yields
\[\Big\langle \big((1\;2), 1, 1\big) \rtimes 1, \big((1\;3), (1\;3), (1\;3)\big) \rtimes 1, \big(1, 1, (1\;2)\big) \rtimes 1, \big(1, (1\;2), 1\big) \rtimes 1 \Big\rangle \approx S_{3} \times S_{3} \times S_{3}.\]
As such, \(\Mon \beta \circ \gamma\) satisfies the short exact sequence
\[1 \rightarrow S_{3} \times S_{3} \times S_{3} \rightarrow \Mon \beta \circ \gamma \rightarrow C_{3} \rightarrow 1.\]
\end{example}

\begin{cor}
\cites[Theorem 3.3]{zvonkinadrianov}[Section 4.3]{bond}
    The monodromy group \(\Mon \beta \circ \gamma\) of the composition of a dynamical Bely\u{\i} map \(\beta\) and a Bely\u{\i} map \(\gamma\) is isomorphic to a subgroup of the wreath product \(\Mon \gamma \wr_{E_{\beta}} \Mon \beta\).  Moreover, this isomorphism is given by
    \begin{align*}
        \Mon \beta \circ \gamma &\longrightarrow \Mon \gamma \wr_{E_{\beta}} \Mon \beta\\
        \rho_{\beta \circ \gamma}(\lambda) &\longmapsto \big(\rho_{\gamma} \circ f_{\lambda}, \rho_{\beta}(\lambda)\big),
    \end{align*}
    where \(f_{\lambda}\) is the extending pattern of \(\beta\) from \(\lambda\).
\end{cor}

\begin{cor}
    \cite[Theorem 4.20]{bond}
    Let \(\beta\) be a dynamical Bely\u{\i} map with monodromy \((\tau_{0}, \tau_{1})\) and extending pattern \((f_{0}, f_{1})\).  Then for any Bely\u{\i} map \(\gamma\),
    \[\Mon \beta \circ \gamma = \big\langle\big(\rho_{\gamma} \circ f_{0}, \tau_{0}\big), \big(\rho_{\gamma} \circ f_{1}, \tau_{1}\big)\big\rangle.\]
    \label{compositiongenerators}
\end{cor}

\begin{example}
Applying Corollary \ref{compositiongenerators} to the extending pattern from Example \ref{extendingexample2}, for \(\beta, \gamma\) defined in \eqref{betadef},
\begin{align*}
\Mon \beta \circ \gamma &= \big\langle (1, 1, \sigma_{1}) \rtimes \tau_{0}, (\sigma_{0}, 1, 1) \rtimes \tau_{1}\big\rangle\\
&= \Big\langle \big(1, 1, (1\;3)\big) \rtimes (1\;2\;3), \big((1\;2), 1, 1\big) \rtimes 1 \Big\rangle\\
&= (S_{3} \times S_{3} \times S_{3}) \rtimes C_{3},
\end{align*}
in agreement with Example 4.7.
\end{example}

\section{Working with the Extended Monodromy Group}

In this section, let \(a, b\) be the generators of \(\pi_{1}(\mathbb{P}^{1}(\mathbb{C})_{*})\), where \(a \simeq_{p} e^{2\pi it}/2\) and \(b \simeq_{p} 1 - e^{2\pi it}/2\).

\subsection{Computing the Extended Monodromy Group}\hspace*{\fill}
\label{computingperms}

Based on Convention \ref{sheetconv} and Lemma \ref{groupoidhom}, determining the extending pattern of a dynamical Bely\u{\i} map \(\beta\) amounts to determining where the liftings \(\beta^{-1}(e^{2\pi it}/2)\) and \(\beta^{-1}(1 - e^{2\pi it}/2)\) cross \((-\infty, 0)\) or \((1, \infty)\) and in which direction.  One approach is to use differential equations to lift the loops \(e^{2\pi it}/2\) and \(1 - e^{2\pi it}/2\) \cite[Section 9.2.5]{goins}, then sample points along the lifted paths and identify consecutive points whose imaginary parts have distinct signs.  The extending pattern is then determined by Description \ref{casebreakdown} according to the identified sign change.

Let \(\lvert E_{\beta} \rvert = n\) and let \(\gamma\) be a Bely\u{\i} map with \(\lvert E_{\gamma} \rvert = m\).  Labeling the edges of \(\beta\), \(\gamma\), and \(\beta \circ \gamma\) beginning with \(0\), the identification \((s, t) \mapsto s \cdot n + t\) gives a mapping \(E_{\gamma} \times E_{\beta} \rightarrow E_{\beta \circ \gamma}\).  In this way, \eqref{extendedaction} specifies the action of a loop \(\lambda\) as
\[(s \cdot n + t)^{\lambda} \longleftrightarrow (s, t)^{\lambda} = \Big(s^{f_{\lambda}(t)}, t^{\lambda}\Big) \longleftrightarrow s^{f_{\lambda}(t)} \cdot n + t^{\lambda}.\]
As such, Algorithm \ref{compositealg} constructs the monodromy of \(\beta \circ \gamma\) as arrays \((\eta_{0}, \eta_{1})\) with \(\eta_{i}[j] = \eta_{i}(j)\) from the extending pattern \((f_{0}, f_{1})\) of \(\beta\) and the monodromies \((\tau_{0}, \tau_{1})\) and \((\sigma_{0}, \sigma_{1})\) of \(\beta\) and \(\gamma\), respectively.

\stepcounter{thm}
\begin{algorithm}[ht]
    \caption{Obtaining the monodromy of \(\beta \circ \gamma\) from \(\beta\) and \(\gamma\)}
    \label{compositealg}
\begin{algorithmic}[1]
    \Function{CompositeMonodromy}{\(\{\tau_{0}, \tau_{1}, f_{0}, f_{1}\}\), \(\{\sigma_{0}, \sigma_{1}\}\)}
    \State \(\rho_{\gamma*} \gets (a \mapsto \sigma_{0}, b \mapsto \sigma_{1})\)
    \State \(f_{0}, f_{1} \gets \rho_{\gamma*}(f_{0}), \rho_{\gamma*}(f_{1})\)
    \State \(n \gets \lvert E_{\beta} \rvert\), \(m \gets \lvert E_{\gamma} \rvert\)
    \State \(\eta_{0} \gets [\,],\, \eta_{1} \gets [\,]\)
    \For{\(0 \leq s < m\)}
        \For{\(0 \leq t < n\)}
            \State append \(f_{0}(t)(s) \cdot n + \tau_{0}(t)\) to \(\eta_{0}\)
            \State append \(f_{1}(t)(s) \cdot n + \tau_{1}(t)\) to \(\eta_{1}\)
        \EndFor
    \EndFor
    \State\Return{\(\eta_{0}, \eta_{1}\)}
\end{algorithmic}
\end{algorithm}

\subsection{Extended Monodromy Group Examples}\hspace*{\fill}

\newsavebox{\boxzero}
\savebox{\boxzero}{\scriptsize\(0\)}
\newsavebox{\boxone}
\savebox{\boxone}{\scriptsize\(1\)}

\stepcounter{thm}
\begin{figure}[ht]
    \centering
    \scalebox{0.55}{
    \begin{tikzpicture}
		\GraphInit[vstyle=Classic]
        \SetVertexMath
        \tikzset{
            VertexStyle/.style = {
                inner sep = 0pt,
                outer sep = 0pt,
                shape = circle,
                fill = black,
                minimum size = 1.5pt,
                draw}
                }
        \Vertex[empty]{A}
        \SetGraphUnit{5.5}
        \WE[empty](A){neg}
        \SetGraphUnit{1}
        \EA[empty](A){pos}
        \SetUpEdge[color = gray!50]
        \Edge(neg)(pos)
        \Vertex[L = \usebox{\boxone}, LabelOut = true, Lpos = 135, Ldist = -0.65cm]{01}
        \tikzset{
            VertexStyle/.style = {
                inner sep = 0pt,
                outer sep = 0pt,
                shape = circle,
                fill = black,
                minimum size = 6pt,
                draw}
                }
        \SetGraphUnit{2.5}
        \tikzset{VertexStyle/.append style={fill = black}}
        \WE[L = \usebox{\boxzero}, LabelOut = true, Lpos = 135, Ldist=-0.65cm](A){00}
        \tikzset{VertexStyle/.append style={fill = white}}
        \SetGraphUnit{2.5}
        \EA[NoLabel = true](00){C}
        \SetGraphUnit{1.7678}
        \NOWE[NoLabel = true](00){B}
        \SOWE[NoLabel = true](00){D}
    \SetUpEdge[labelstyle = {sloped,above = 2}]
    	\Edge[label = 2](00)(B)
        \Edge[label = 3](00)(D)
    \SetUpEdge[labelstyle = {sloped,below = 2}]        
        \Edge[label = 1](00)(C)
    \end{tikzpicture}
    }\qquad\qquad
    \scalebox{0.55}{
    \begin{tikzpicture}
		\GraphInit[vstyle=Classic]
        \SetVertexMath
        \tikzset{
            VertexStyle/.style = {
                inner sep = 0pt,
                outer sep = 0pt,
                shape = circle,
                fill = black,
                minimum size = 1.5pt,
                draw}
                }
        \Vertex[empty]{A}
        \SetGraphUnit{5.5}
        \EA[empty](A){pos}
        \SetGraphUnit{1}
        \WE[empty](A){neg}
        \SetUpEdge[color = gray!50]
        \Edge(neg)(pos)
        \Vertex[L = \usebox{\boxzero}, LabelOut = true, Lpos = 135, Ldist = -0.65cm]{00}
        \tikzset{
            VertexStyle/.style = {
                inner sep = 0pt,
                outer sep = 0pt,
                shape = circle,
                fill = black,
                minimum size = 6pt,
                draw}
                }
        \SetGraphUnit{2.5}
        \tikzset{VertexStyle/.append style={fill = black}}
        \EA[L = \usebox{\boxone}, LabelOut = true, Lpos = 135, Ldist=-0.65cm](A){01}
        \tikzset{VertexStyle/.append style={fill = white}}
        \SetGraphUnit{2.5}
        \WE[NoLabel = true](01){C}
        \SetGraphUnit{1.7678}
        \NOEA[NoLabel = true](01){B}
        \SOEA[NoLabel = true](01){D}
    \SetUpEdge[labelstyle = {sloped,below = 2}]
    	\Edge[label = 3](01)(B)
        \Edge[label = 2](01)(D)
    \SetUpEdge[labelstyle = {sloped,above = 2}]        
        \Edge[label = 1](01)(C)
    \end{tikzpicture}
    }
\caption{The dessins of \(\beta_{1}\) (left) and \(\beta_{2}\) (right).}
\label{beta1beta2}
\end{figure}
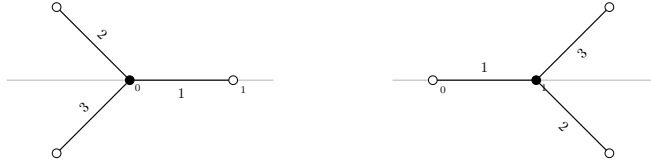

\smallskip
\begin{example}
Consider \(\beta_{1}\) and \(\beta_{2}\) as in the introduction, which have monodromy \((\tau_{0}, \tau_{1}) = \big((1\;2\;3), (1)\big)\).  The extending patterns are given by
\begin{align*}
    f_{0}^{(1)} = (1, a, 1),\quad f_{1}^{(1)} = (b, 1, 1),\qquad f_{0}^{(2)} = (1, b, 1),\quad f_{1}^{(2)} = (a, 1, 1),
\end{align*}
as can be seen from Figure \ref{beta1beta2}.  For \(\gamma\), \((\sigma_{0}, \sigma_{1}) = \big((1\;2), (1)\big)\), so that in the case of \(\beta_{1}\),
\begin{align*}
    \rho_{\gamma*} \circ \varphi_{\beta_{1}}(a) &= \rho_{\gamma*}\big((1, a, 1) \rtimes \tau_{0}\big) = (1, \sigma_{0}, 1) \rtimes \tau_{0},\\
    \rho_{\gamma*} \circ \varphi_{\beta_{1}}(b) &= \rho_{\gamma*}\big((b, 1, 1) \rtimes \tau_{1}\big) = (1, 1, 1) \rtimes 1.
\end{align*}
Computing powers of \((1, \sigma_{0}, 1) \rtimes \tau_{0}\) shows that
\[\Mon \beta_{1} \circ \gamma = \langle(1, \sigma_{0}, 1) \rtimes \tau_{0} \rangle \approx C_{6}.\]
On the other hand, for \(\beta_{2}\),
\begin{align*}
    \rho_{\gamma*} \circ \varphi_{\beta_{2}}(a) &= \rho_{\gamma*}\big((1, b, 1) \rtimes \tau_{0}\big) = (1, 1, 1) \rtimes \tau_{0},\\
    \rho_{\gamma*} \circ \varphi_{\beta_{2}}(b) &= \rho_{\gamma*}\big((a, 1, 1) \rtimes \tau_{1}\big) = (\sigma_{0}, 1, 1) \rtimes 1,
\end{align*}
and \(\Mon \beta_{2} \circ \gamma \approx C_{2} \times A_{4}\).
\end{example}

\smallskip
\begin{example}
Let
\begin{gather*}
    \beta_{0}(z) \coloneqq \frac{(z^4+228z^3+494z^2-228z+1)^3}{1728z(z^2-11z-1)^5},\\
    \mu(z) \coloneqq \frac{55\sqrt{5}+123}{5\sqrt{5}+11}z = (11+\alpha^{-1})z,
\end{gather*}
where \(\alpha\) is the positive root of \(z^{2}-11z-1\).  Let \(\beta(z) = \beta_{0} \circ \mu(z)\), and consider the family of Bely\u{\i} maps \(\gamma_{m}(z) = z^{m}\).  Because the pole of \(\beta_{0}\) lying in the face of \(\beta\) containing \(z = 1\) does not lie over \(1\), \(\mu\) is needed to move the pole to \(z = 1\) to make \(\beta\) dynamical.  The monodromy and extending pattern of \(\beta\) under the ordering of edges given in Figure \ref{circledessin} is given by

\begin{equation*}
    \begin{gathered}
    \tau_{0} = (1\;2\;3)(4\;5\;6)(7\;9\;8)(10\;11\;12),\\
    \tau_{1} = (1\;2)(3\;4)(5\;8)(6\;7)(9\;10)(11\;12),
\end{gathered} \qquad
\begin{gathered}
    f_{0} = [a, a^{-1}, 1, a, a^{-1}, 1, 1, 1, 1, 1, 1, 1],\\
    f_{1} = [b^{-1}, b, 1, 1, 1, 1, 1, 1, 1, 1, a, a^{-1}].
\end{gathered}
\end{equation*}

\stepcounter{thm}
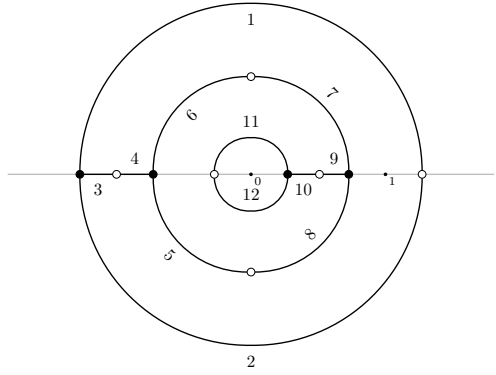
\begin{figure}[ht]
    \centering
    \scalebox{0.65}{
    \begin{tikzpicture}
		\GraphInit[vstyle=Classic]
        \SetVertexMath
        \tikzset{
            VertexStyle/.style = {
                inner sep = 0pt,
                outer sep = 0pt,
                shape = circle,
                fill = black,
                minimum size = 1.5pt,
                draw}
                }
        \Vertex[empty]{A}
        \SetGraphUnit{5}
        \EA[empty](A){pos}
        \WE[empty](A){neg}
        \SetUpEdge[color = gray!50]
        \Edge(neg)(pos)
        \Vertex[L = \usebox{\boxzero}, LabelOut = true, Lpos = 135, Ldist=-0.5cm]{00}
        \SetGraphUnit{2.75}
        \EA[L = \usebox{\boxone}, LabelOut = true, Lpos = 135, Ldist = -0.5cm](A){01}
        \tikzset{
            VertexStyle/.style = {
                inner sep = 0pt,
                outer sep = 0pt,
                shape = circle,
                fill = black,
                minimum size = 4.5pt,
                draw}
                }
        \SetGraphUnit{0.75}
        \tikzset{VertexStyle/.append style={fill = black}}
        \EA[NoLabel = true](A){B}
        \tikzset{VertexStyle/.append style={fill = white}}
        \WE[NoLabel = true](A){C}
        \SetGraphUnit{0.65}
        \EA[NoLabel = true](B){D}
        \SetGraphUnit{2}
        \tikzset{VertexStyle/.append style={fill = black}}
        \EA[NoLabel = true](A){E}
        \WE[NoLabel = true](A){F}
        \tikzset{VertexStyle/.append style={fill = white}}
        \NO[NoLabel = true](A){G}
        \SO[NoLabel = true](A){H}
        \SetGraphUnit{0.75}
        \WE[NoLabel = true](F){I}
        \SetGraphUnit{3.5}
        \EA[NoLabel = true](A){J}
        \tikzset{VertexStyle/.append style={fill = black}}
        \WE[NoLabel = true](A){K}
        \tikzset{VertexStyle/.style={shape=coordinate}}
        \NO[Lpos=270,Ldist=0.1cm](A){1}
        \SO[Lpos=270,Ldist=0.1cm](A){2}
        \SetGraphUnit{0.75}
        \NO[Lpos=90,Ldist=0.1cm](A){11}
        \SO[Lpos=90,Ldist=0.1cm](A){12}
    \SetUpEdge[labelstyle = {sloped,below = 2}]
    	\Edge[label = 10](B)(D)
        \Edge[label = 3](I)(K)
    \SetUpEdge[labelstyle = {sloped,above = 2}]
        \Edge[label = 9](D)(E)
        \Edge[label = 4](F)(I)
        \SetUpEdge[labelstyle = {sloped,above = 2}]
        \tikzset{EdgeStyle/.append style={bend left=45}}
        \Edge[label = 8](E)(H)
        \SetUpEdge[labelstyle = {sloped,below = 2}]
        \tikzset{EdgeStyle/.append style={bend left=45}}
        \Edge[label = 5](H)(F)
        \SetUpEdge[labelstyle = {sloped,above = 2}]
        \tikzset{EdgeStyle/.append style={bend right=45}}
        \Edge[label = 7](E)(G)
        \SetUpEdge[labelstyle = {sloped,below = 2}]
        \tikzset{EdgeStyle/.append style={bend right=45}}
        \Edge[label = 6](G)(F)
        \tikzset{EdgeStyle/.append style={bend left=45}}
        \Edge(J)(2)
        \Edge(2)(K)
        \Edge(B)(12)
        \Edge(12)(C)
        \tikzset{EdgeStyle/.append style={bend right=45}}
        \Edge(J)(1)
        \Edge(1)(K)
        \Edge(B)(11)
        \Edge(11)(C)
    \end{tikzpicture}
    }
\caption{A plot of \(\beta\) with edges labeled}
\label{circledessin}
\end{figure}

Computing \(\MonExt \beta\) using GAP, as in Appendix \ref{app}, finds that
\begin{equation*}
\begin{array}{r@{\extracolsep{\fill}}l@{\hskip 2pt}r@{\hskip 4pt}r@{\hskip 4pt}r@{\hskip 4pt}r@{\hskip 4pt}r@{\hskip 4pt}r@{\hskip 4pt}r@{\hskip 4pt}r@{\hskip 4pt}r@{\hskip 4pt}r@{\hskip 4pt}r@{\hskip 4pt}r@{\hskip 2pt}l}
\MonExt \beta = \langle&[&a^{-5}, &1, &1, &1, &1, &1, &1, &1, &1, &1, &a^{5}, &1&],\\ 
&[&1, &a^{-5}, &1, &1, &1, &1, &1, &1, &1, &1, &1, &a^{5}&],\\ 
&[&1, &1, &a^{-5}, &1, &1, &1, &1, &1, &1, &a^{5}, &1, &1&],\\ 
&[&1, &1, &1, &a^{-5}, &1, &1, &1, &1, &a^{5}, &1, &1, &1&],\\ 
&[&1, &1, &1, &1, &a^{-5}, &1, &1, &a^{5}, &1, &1, &1, &1&],\\ 
&[&1, &1, &1, &1, &1, &a^{-5}, &a^{5}, &1, &1, &1, &1, &1&]\rangle.
\end{array}
\end{equation*}

\noindent It follows that \(\Mon \beta \circ \gamma_{m}\) is an extension of
\[\MonExt \beta \approx \begin{cases}
C_{m}^{6} & \text{if \(5 \centernot\mid m\),}\\
C_{m/5}^{6} & \text{if \(5 \mid m\),}
\end{cases}\]
by \(\Mon \beta = A_{5}\).  Note that by the Schur-Zassenhaus theorem, if \(\gcd(m, 30) = 1\) or if \(m = 5\), then this extension is guaranteed to be split.
\end{example}

\smallskip
\begin{example}
Consider once again \(\beta\) from the previous example and let \(\gamma\) be a Bely\u{\i} map with monodromy given by \((\sigma_{0}, \sigma_{1}) = \big((2\;3\;4), (1\;2)(3\;4)\big)\).  Then
\[\MonExt \beta = C_{2}^{6} \rtimes C_{3},\qquad \Mon \beta = C_{3},\qquad \Mon \beta \circ \gamma = C_{2}^{6} \rtimes C_{9}.\]
But since \(C_{9}\) is not a split extension of \(C_{3}\) by \(C_{3}\), this shows that \(\Mon \beta \circ \gamma\) is not in general a split extension of \(\MonExt \beta\).
\end{example}

\section{Conclusion}

By using Theorem \ref{actiononbetagamma} to characterize the action of \(\pi_{1}^{Z}\) on  sheets lying over \(Y\) in terms of functions into \(\pi_{1}^{Y}\), it is possible to capture the information about \(\beta\) required to determine \(\Mon \beta \circ \gamma\) for any Bely\u{\i} map \(\gamma\).  The concept of the extending pattern of a dynamical Bely\u{\i} map, a pair of functions \(f_{0}, f_{1}: E_{\beta} \rightarrow \pi_{1}^{Y}\), is introduced to express the action of \(\pi_{1}^{Z}\), when lifted by \(\beta\), on sheets lying over \(Y\).  Finally, the extending pattern is used to determine the group which \(\Mon \beta\) extends to \(\Mon \beta \circ \gamma\).

The dynamical nature of \(\beta\) is primarily applied in the use of sheets of \(\gamma\) of a prescribed form to identify how paths between edges of \(\beta \circ \gamma\) affect the edge of \(\gamma\) corresponding to an edge of \(\beta \circ \gamma\).  Extending the approach of Section \ref{pathsonsheets} to arbitrary Bely\u{\i} maps would allow for applying Section \ref{monwreath} to the determination of \(\Mon \beta \circ \gamma\) for non-dynamical Bely\u{\i} maps, as well as additional classes of covering maps such as origamis.

Finally, although the extending pattern of a dynamical Bely\u{\i} map \(\beta\) permits determination of \(\MonExt \beta\), so that once a Bely\u{\i} map is specified, the groups which extend to \(\Mon \beta \circ \gamma\) can be determined, it would be useful to devise conditions on \(\beta\) and \(\gamma\) to determine when \(\Mon \beta \circ \gamma\) is a split extension of \(\MonExt \beta\) and \(\Mon \beta\), as well as when \(\Mon \beta \circ \gamma\) is a proper subgroup of \(\Mon \gamma \wr_{E_{\beta}} \Mon \beta\).

\printbibliography

\appendix
\section{Computing Composite Monodromies with GAP}
\label{app}

\subsection{Computing \texorpdfstring{\(\MonExt \beta\)}{MonExt B}}\hspace*{\fill}

To compute \(\EMon \beta\), along with \(\Mon \beta \circ \gamma_{m}\), in Example 5.3, SageMath's \cite{sage} interface to GAP \cite{gap} is used.  First, \(\EMon \beta\) is constructed using the monodromy and extending pattern of \(\beta\).

\smallskip
\begin{small}
\noindent\texttt{sage: F2 = libgap.FreeGroup(\textquotesingle a\textquotesingle, \textquotesingle b\textquotesingle); A,B = F2.GeneratorsOfGroup()\\
sage: tau0 = libgap.eval(\textquotesingle(1,2,3)(4,6,5)(7,9,8)(10,11,12)\textquotesingle)\\
sage: tau1 = libgap.eval(\textquotesingle(1,2)(3,4)(5,7)(6,8)(9,10)(11,12)\textquotesingle)\\
sage: MonBeta = libgap.Group(tau0, tau1)\\
sage: rho = F2.GroupHomomorphismByImages(MonBeta)\\
sage: wr = F2.WreathProduct(MonBeta)\\
sage: a\char`\_gens = [A\^{}wr.Embedding(j) for j in range(1, 13)]\\
sage: b\char`\_gens = [B\^{}wr.Embedding(j) for j in range(1, 13)]
}\end{small}

\medskip
\noindent Because \(\rho_{\gamma}(b) = 1\), \(f_{1}(1)\) and \(f_{1}(2)\) are set to \(1\), rather than \(b\) and \(b^{-1}\), respectively, prior to application of \(\rho_{\gamma*}\) in order to simplify the computations.  Specifically, the goal in this computation is the determination of the kernel of \(\proj_{\beta}: \Mon \beta \circ \gamma \rightarrow \Mon \beta\) in the case that \(\rho_{\gamma}(a) = (1\;\cdots\;m)\) and \(\rho_{\gamma}(b) = 1\).  But 
\(\rho_{\gamma*}\rtimes \id = (\rho_{\gamma*} \rtimes \id) \circ (a \mapsto a, b \mapsto 1)\),
so that their kernels are equal.

\smallskip
\begin{small}
\noindent\texttt{sage: f0 = a\char`\_gens[0] * a\char`\_gens[1]\^{}-1 * a\char`\_gens[3] * a\char`\_gens[5]\^{}-1\\
sage: f1 = a\char`\_gens[10] * a\char`\_gens[11]\^{}-1\\
sage: tau0, tau1 = wr.Embedding(13).Image().GeneratorsOfGroup()\\
sage: EMonBeta = libgap.Group(f0*tau0, f1*tau1)
}\end{small}

\medskip
\noindent Next, \(\MonExt \beta\) is computed using Lemma \ref{projbeta}.  As a result of a lack of efficient methods for finding generators in a wreath product not expressed as a permutation group, GAP is unable to find a minimal generating set for \texttt{KerProjBeta}, instead finding a set of \(61\) generators.  However, removing duplicates from this set reduces the generating set to \(15\) generators.

\smallskip
\begin{small}
\noindent\texttt{sage: phi = F2.GroupHomomorphismByImages(EMonBeta)\\
sage: KerProjBeta = phi.RestrictedMapping(rho.Kernel()).Image()\\
sage: gens = list(KerProjBeta.GeneratorsOfGroup().Unique())
}\end{small}

\medskip
\noindent Finally, from a set \(\{g_{i}\}_{i = 1}^{6}\) of six generators, all \(15\) unique generators can be shown to have the form \(\prod_{i = 1}^{6} g_{i}^{e_{i}}\), with \(e_{i} \in \{0, -1, 1\}\).  Manual inspection of \(\{g_{i}\}_{i = 1}^{6}\) shows them to be a minimal generating set.\footnote{The output of the final command was formatted to improve readability.}

\smallskip
\begin{small}
\noindent\texttt{sage: exponents\char`\_0\char`\_pm\char`\_1 = ( (wr.One(), gen, gen\^{}-1) for gen in gens[1:7] )\\
sage: gen\char`\_and\char`\_inv\char`\_prods = map(prod, itertools.product(*exponents\char`\_0\char`\_pm\char`\_1))\\
sage: set(gens).issubset(gen\char`\_and\char`\_inv\char`\_prods)\\
True\\
sage: gens[1:7]\\
\noindent[WreathProductElement(<id>, a\^{}5,  <id>, <id>, <id>, <id>,\\
\hphantom{[WreathProductElement(}<id>, <id>, <id>, <id>, <id>, a\^{}-5, ()),\\
\hphantom{[}WreathProductElement(a\^{}-5, <id>, <id>, <id>, <id>, <id>,\\
\hphantom{[WreathProductElement(}<id>, <id>, <id>, <id>, a\^{}5,  <id>, ()),\\
\hphantom{[}WreathProductElement(<id>, <id>, a\^{}-5, <id>, <id>, <id>,\\
\hphantom{[WreathProductElement(}<id>, <id>, <id>, a\^{}5,  <id>, <id>, ()),\\
\hphantom{[}WreathProductElement(<id>, <id>, <id>, a\^{}-5, <id>, <id>,\\
\hphantom{[WreathProductElement(}<id>, <id>, a\^{}5,  <id>, <id>, <id>, ()),\\
\hphantom{[}WreathProductElement(<id>, <id>, <id>, <id>, <id>, a\^{}-5,\\
\hphantom{[WreathProductElement(}a\^{}5,  <id>, <id>, <id>, <id>, <id>, ()),\\
\hphantom{[}WreathProductElement(<id>, <id>, <id>, <id>, a\^{}-5, <id>,\\
\hphantom{[WreathProductElement(}<id>, a\^{}5,  <id>, <id>, <id>, <id>, ())]
}\end{small}

It is worth noting that in the case \(\EMon \beta\) can be viewed as a subgroup of \(\mathbb{Z} \wr_{E_{\beta}} \Mon \beta\), as in this example where \(\EMon \beta\) is specialized under the assumption \(\sigma_{1}^{\gamma} = 1\) or in the case that the edges of \(\beta\) cross only one of \((-\infty, 0)\) and \((1, \infty)\), the use of Residue-Class-Wise Affine groups \cite{rcwa} provides an efficient method for determination of generators of \(\MonExt \beta\).
\end{document}